\documentclass[a4paper,12pt]{article}

\usepackage{graphicx}
\usepackage{amsmath,amsthm}
\usepackage{amsfonts,amssymb}
\usepackage{amscd}
\usepackage{color}
\usepackage[cp1251]{inputenc}
\usepackage[matrix,arrow,curve]{xy}
\usepackage{tikz}
\usetikzlibrary{arrows}
\usepackage{mathtools}

\newtheorem{lemma}{Lemma}[subsection]
\newtheorem{prop}[lemma]{Proposition}
\newtheorem{theorem}[lemma]{Theorem}
\newtheorem{coro}[lemma]{Corollary}
\newtheorem{rema}[lemma]{{\rm\bf Remark}}
\newtheorem{Def}[lemma]{Definition}
\newtheorem{Ques}{Problem}

\newtheorem{ex}[lemma]{Example}

\newcommand{\Lin}{{lr\rm Lin}}

\newcommand{\Hom}{{\rm Hom}}

\newcommand{\charr}{{\rm char\,}}
\newcommand{\gd}{{\rm gldim\,}}
\newcommand{\kk}{{\bf k}}

\DeclareMathOperator{\Ker}{{Ker}}

\newcommand{\Ext}{{\rm Ext}}

\newcommand{\Ann}{{\rm Ann}}

\renewcommand{\Im}{{\rm Im\,}}

\newcommand{\HH}{{\mathcal H}}

\newcommand{\FF}{{\mathcal F}}
\newcommand{\GG}{{\mathcal G}}
\newcommand{\TT}{{\mathcal T}}

\newcommand{\bZ}{\mathbb{Z}}

\newcommand{\bD}{{\bf D}}

\newcommand{\vp}{\varphi}

\newcommand{\nds}{{\not\subset}^{\oplus}}
\newcommand{\ds}{\subset^{\oplus}}

\newcommand{\ot}{\otimes}

\renewcommand{\le}{\leqslant}
\renewcommand{\ge}{\geqslant}

\newenvironment{Ex}{\begin{ex}
\rm }{\flushright
$\Box$\end{ex}}

\numberwithin{equation}{subsection}
\begin{document}

\title{Homogeneous triples for homogeneous algebras with two relations.}

\author{Eduardo do Nascimento Marcos, Yury Volkov 
\footnote{The first named author was supported by the tematic project of Fapesp number  2014/09310-5.
The second  named author was supported by the RFBR research project number 18-31-20004
and had a trip to S\~ao Paulo paid by the Fapesp process number 2018/07046-0 during working on this paper.}}

\maketitle
\begin{abstract}
In our preceding paper we have introduced the notion of an $s$-homogeneous triple. In this paper we use this technique to study connected $s$-homogeneous algebras with two relations.
For such algebras, we describe all possible pairs $(A,M)$, where $A$ is the $s$-Veronese ring and $M$ is the $(s,1)$-Veronese bimodule of the $s$-homogeneous dual algebra.
For each such a pair we give an intrinsic characterization of algebras corresponding to it. Due to results of our previous work many pairs determine the algebra uniquely up to isomorphism.
Using our partial classification, we show that, to check the $s$-Koszulity of a connected $s$-homogeneous algebras with two relations, it is enough to verify an equality for Hilbert series or to check the exactness of the generalized Koszul complex in the second term.
For each pair $(A,M)$ not belonging to one specific series of pairs, we check if there exists an $s$-Koszulity algebra corresponding to it.
Thus, we describe a class of possible $\Ext$-algebras of $s$-Koszul connected algebras with two relations and realize all of them except a finite number of specific algebras as $\Ext$-algebras.
Another result that follows from our classification is that an $s$-homogeneous algebra with two dimensional $s$-th component cannot be $s$-Koszul for $s>2$.
\end{abstract}

\section{Introduction}

All the algebras under consideration are graded algebras of the form $\Lambda = T_\kk V/I$, where $\kk$ is an algebraically closed field, $T_\kk V$ is the tensor algebra of some finite dimensional space $V$
with the standard grading, and $I$ is an ideal generated by some subspace $W\subset T_\kk V$ such that all the elements of $W$ are homogeneous and the same degree.
The case $\dim_{\kk}W=1$ was studied in \cite{Dic} (see also \cite{B2, MV}). This paper deals with the case $\dim_{\kk}W=2$ whose consideration was initiated in \cite{MV}.
For this purpose we will use the technique of $s$-homogeneous triples introduced in the last mentioned paper.

The $\Ext$-algebra of the algebra $\Lambda$ as above is the graded algebra $\bigoplus_{i\geq 0} \Ext^i_{\Lambda}(\kk, \kk)$.
The notion of a Koszul algebra was introduced by S.~Priddy  in \cite{P}. All Koszul algebras are quadratic and they appear in 
pairs. The $\Ext$-algebra of a Koszul algebra $\Lambda$, is another time a Koszul algebra and its $\Ext$-algebra is the original
one. If we know that an algebra $\Gamma$ is the $\Ext$-algebra of a Koszul algebra $\Lambda$ then one may
recuperate $\Lambda$ from $\Gamma$ by taking its quadratic dual. Note that even if a quadratic algebra is not Koszul, then it can be recovered from its quadratic dual algebra that in the non Koszul case is not anymore isomorphic to the $\Ext$-algebra. Note also that if we have a quadratic algebra, then it is an $\Ext$-algebra of some quadratic algebra if and only if it is Koszul.

The notion of a Koszul algebra was generalized to the $s$-homogeneous case by Berger in \cite{B} for connected algebras.
Later the definition was rewritten for the case of an arbitrary quiver in \cite{GMMZ}.
This notion turned out to be important.
For example, it was shown in \cite{BS} that if $\Lambda$ is an algebra defined by a homogeneous potential of degree $s+1$ and $\dim_{\kk} W=\dim_{\kk} V$, then $\Lambda$ is 3-Calabi-Yau if and only if it is $s$-Koszul and $\dim_{\kk}\big(W\ot_{\kk}V\cap V\ot_{\kk}W\big)=1$.

Calabi-Yau algebras  were introduced in \cite{Ginz} while it is difficult to say where their twisted generalization that we will consider in the current paper was introduced. Twisted Calabi-Yau algebras were considered in many papers, see \cite{BS, BSW, GK, LR, LWW, RRZ}. The connections between Calabi-Yau algebras and algebras defined by a potential was noted in \cite{Ginz}, where it was proved that $3$-Calabi-Yau algebras that appear in some sense in a natural way are algebras defined by a potential. It was proved in \cite{Boc} that any graded $3$-Calabi-Yau algebra is defined by a potential. Recently, quadratic $3$-Calabi-Yau algebras with three generators and cubic $3$-Calabi-Yau algebras with two generators were classified in \cite{MS,MU}. All the algebras with the same number of generators and relations defined by a twisted potential were classified in \cite{Iud}. Twisted potentials that will appear in this paper were considered in \cite{BSW}. It is proved in \cite{RRZ} that twisted $3$-Calabi-Yau algebras are exactly Artin-Schelter regular algebras of dimension $3$. These algebras were introduced in \cite{AS} and were studied in many works since then. In particular, it was shown in \cite{Lev} that such algebras are always domains.

In this  paper an element of the free algebra $T_\kk V$ will be called a polynomial.

The main objective of our work \cite{MV} was to give a method of recovering an $s$-Koszul algebra from its $\Ext$-algebra or, more generally, recover an $s$-homogeneous algebra from the associated {\it $s$-homogeneous triple}. The last notion was introduced in the same paper and gives the $\Ext$-algebra with its $A_{\infty}$-structure in the $s$-Koszul case.

Using these ideas, we have shown in \cite{MV} that the algebra $\Lambda=\kk\langle x_1,\dots,x_m\rangle/(f)$, where $f$ is a polynomial of degree $s$, is not $s$-Koszul if and only if $f\not=g^s$ for any linear polynomial $g$ and there are some polynomial $g_1$, $g_2$ and $h$ of nonzero degree such that $f=g_1h=hg_2$. We have also shown that, in the case $s>2$, the algebra  $\Lambda^!$ is $s$-Koszul only in the case $m=1$ (i.e. in the case $\Lambda=\kk[x]/(x^s)$). Moreover if $\Lambda=\kk\langle x_1,\dots,x_m\rangle/(f)$ is $s$-Koszul and $f$ is not of the form $g^s$ for some linear polynomial $g$  then the global dimension of $\Lambda$ is two. If $\Lambda=\kk\langle x_1,\dots,x_m\rangle/(g^s)$ with $g$ a linear polynomial then $\Lambda$ has infinite global dimension. All these results in fact can be deduced from \cite{Dic}, where the Poincar\'e series of algebras with one relation were computed and, in particular, it was shown that the global dimension of $\Lambda$ is infinite in the not $s$-Koszul case.

In the current paper we obtain analogous results for an algebra $\Lambda$ of the form $\Lambda=\kk\langle x_1,\dots,x_m\rangle/(f_1, f_2)$, where $f_1$ and $f_2$ are linearly independent polynomials  of degree $s$. The study of such algebras was initiated in \cite{MV}, where it was shown that the $s$-Veronese ring of $\Lambda^!$ is a quadratic algebra with two generators and at least two relations and the algebras $\Lambda$, for which this ring is isomorphic to $\kk\langle x,y\rangle/(xy,yx)$ or $\kk\langle x,y\rangle/(x^2,y^2)$ were classified. In the last mentioned cases $\Lambda$ is $s$-Koszul of infinite global dimension and $\Lambda^!$ is not $s$-Koszul except the case $s=2$.
Here we will show that the last fact is not accidental and $\Lambda^!$ is not $s$-Koszul if $s>2$ and $\Lambda$ is an $s$-homogeneous algebra defined by two relations (see Theorem \ref{nkoz}). Our main result is the classification of pairs $\big((\Lambda^!)^{(s)}),(\Lambda^!)^{(s,1)}\big)$ that appear for $s$-homogeneous algebras $\Lambda$ defined by two relations. Here ${\Lambda}^{(s)}$ and ${\Lambda}^{(s,1)}$ denote the $s$-Veronese ring and the $(s,1)$-Veronese bimodule of the graded algebra ${\Lambda}$.
For the remaining pairs we give an intrinsic description of $s$-homogeneous algebras corresponding to them. Also we discuss the Koszul property of the obtained $s$-homogeneous algebras.
Some  of such pairs determine the corresponding algebra up to isomorphism and we give descriptions of corresponding algebras in such cases.
For the convenience of the reader we collect all our classification results in the last section, where we give also some interesting consequences and ask questions related to the subject. Thus, if one is interested in the final form of our results, he can go directly to Subsection \ref{main_res}.
One of the results we obtain here is the description of all possible $\Ext$-algebras of $s$-Koszul algebras. Nevertheless for each $s$ there is a finite number $s$-homogeneous algebras whose $s$-Koszulity we will not check in this paper. Each of these algebras has a pair $\big((\Lambda^!)^{(s)}),(\Lambda^!)^{(s,1)}\big)$ that determines the original algebra up to isomorphism, and thus we are not able to check if such a pair can be realized as an $\Ext$-algebra. 
  During our manipulations, we will give many interesting examples of $s$-Koszul algebras. In particular, these examples include many $3$-Calabi-Yau and twisted $3$-Calabi-Yau algebras algebras defined by a twisted potential.
Note that the classification of Koszul algebras with two relations is presented in \cite[Theorem 3.10.6]{Ufn} (see also \cite{Bac,Ufn0}) and essentially we make a big step towards the analogous classification in the case where $s>2$. Thus, though major part of our results can be applied to study Koszul algebras too, we will consider only  the case $s>2$.

\section{Preliminary results}

In this section we recall some definitions and results concerning $s$-homogeneous algebras, $s$-homogeneous triples, $s$-Koszul algebras, $3$-Calabi-Yau algebras and algebras defined by a potential, give the classification of quadratic algebras with two generators and recall the classification of indecomposable modules over them. In the last part of this section we describe the general ideas that we will use.

\subsection{$s$-homogeneous algebras and $s$-homogeneous triples}

We fix some notation during the paper. First of all, we fix some algebraically closed ground field $\kk$. All the algebras and vector spaces  in this paper are over $\kk$. So we write simply $\otimes$ instead of $\otimes_\kk$. If $V$ is a graded space, then $V_i$ denotes the subspace of $V$ formed by the elements of degree $i$. Also $V[t]$ denotes the shift of $V$ by $t$, i.e. the graded space that coincides with $V$ as a nongraded space and has the grading defined by the equality $V[t]_i=V_{i+t}$. All the modules in this paper are right modules.
In fact, we will need often to deal with bimodules and with left and right module structures on them, but to talk about the left structure we will simply consider a module over the opposite algebra and will add ${}^{\rm op}$ to the notation of the  algebra. On the other hand, if $M$ is an $A$-bimodule, then we will write $M_A$ to emphasize that we are talking about its right structure and ${}_AM$ for the left structure.
If $A$ is an algebra, $M$ is an $A$-module, and $X$ is a subset of $M$, then $\langle X\rangle_A$ denotes the $A$-submodule of $M$ generated by the set $X$. All graded algebras will be nonnegatively graded by $\mathbb{Z}$.
All the modules that we will consider in this paper will be finitely generated and we will often assume this condition without mentioning.

We will use the standard grading on $\kk\langle x_1,\dots,x_m\rangle$ defined by the equality $\deg(x_i)=1$ ($1\le i\le m$). Throughout the paper we consider the algebra $\Lambda=\kk\langle x_1,\dots,x_m\rangle/I$, where $I$ is some ideal of $\kk\langle x_1,\dots,x_m\rangle$ generated by polynomials of the fixed degree $s> 2$. Such an algebra is called an {\it $s$-homogeneous connected algebra}.
Let $x_1^*,\dots,x_m^*$ be the basis of $\Lambda_1^*=\Hom_\kk(\Lambda_1,\kk)$ dual to the basis
$x_1,\dots,x_m$ of $\Lambda_1$. Then the {\it $s$-homogeneous dual} of $\Lambda$ is the algebra $\Lambda^!=\kk\langle x_1^*,\dots,x_m^*\rangle/I^!$, where $I^!$ is generated by such $f\in \kk\langle x_1^*,\dots,x_m^*\rangle_s\cong \kk\langle x_1,\dots,x_m\rangle_s^*$ that $f(I_s)=0$.
Here we use the grading on $\kk\langle x_1^*,\dots,x_m^*\rangle$ defined by the equalities $\deg(x_i^*)=1$ ($1\le i\le m$). Note that the algebras $\Lambda$ and $\Lambda^!$ are graded in the natural way and that $\Lambda^!$ is an $s$-homogeneous connected algebra. Moreover, $(\Lambda^!)^!\cong\Lambda$.

\begin{Def}{\rm 
Given a graded algebra $A$, we denote by $A^{(r)}$ the {\it $r$-Veronese ring} of $A$, i.e. the graded algebra $A^{(r)}=\oplus_{i\ge 0}A^{(r)}_i$, where $A^{(r)}_i=A_{ri}$ and the multiplication of $A^{(r)}$ is induced by the multiplication of $A$.
Also we define the {\it $(r,t)$-Veronese bimodule} $A^{(r,t)}$ as a graded $A^{(r)}$-bimodule $A^{(r,t)}=\oplus_{i\ge 0}A^{(r,t)}_i$, where $A^{(r,t)}_i=A_{ri+t}$ and the $A^{(r)}$-bimodule structure on $A^{(r,t)}$ is induced by the multiplication of $A$.
Note that the multiplication of $A$ induces an $A^{(r)}$-bimodule homomorphism from $\left(A^{(r,1)}\right)^{\ot_{A^{(r)}}r}$ to $A^{(r)}$. We denote this homomorphism by $\phi_{A}^{(r)}$. 
}
\end{Def}

Given a graded algebra $A$, a graded $A$-module $M$ is called {\it linear up to the $n$-th degree} if there exists  a projective resolution of $M$ in the category of graded $A$-modules
$$
M\leftarrow P_0\leftarrow P_1\leftarrow\cdots\leftarrow P_n\leftarrow\cdots
$$
such that $P_i$ is generated in degree $i$, i.e. $P_i=(P_i)_iA$, for $0\le i\le n$. We will say that $M$ is {\it linearly presented} if it is linear up to the first degree.
We will denote by $J(A)$ the ideal $\oplus_{i>0}A_i$ of $A$.

We define next the category of $s$-homogeneous triples, which will be useful for us.

\begin{Def}\label{str}
{\rm
An {\it $s$-homogeneous  triple} is a triple $(A,M,\vp)$, where $A$ is a quadratic  $\kk$-algebra, $M$ is a graded $A$-bimodule linearly presented over $A$ and over $A^{\rm op}$, and $\vp:M^{\ot^{s}_A }\rightarrow A[1]$ is a homomorphism of graded $A$-bimodules such that 
\begin{enumerate}
\item $\Im\vp=J(A)[1]$;
\item $1_M\ot_A\vp=\vp\ot_A 1_M:M^{\ot_A^{s+1}}\rightarrow M[1]$;
\item $\Ker(1_M\ot_A\vp)=\Ker\vp\ot_A M+M\ot_A\Ker\vp$;
\item $\Ker(\vp\ot_A\vp)=\sum\limits_{i=0}^sM^{\ot_A^i}\ot_A\Ker\vp\ot_A M^{\ot_A^{s-i}}$.
\end{enumerate}
}
\end{Def}

Let $(A,M,\vp)$ and $(B,L,\psi)$ be $s$-homogeneous  triples. A {\it morphism} from $(A,M,\vp)$ to $(B,L,\psi)$ is a pair $(f,g)$, where $f:A\rightarrow B$ is a morphism of graded algebras and $g:M\rightarrow L$ is a morphism of graded $A$-bimodules such that $f\vp=\psi g^{\ot s}$. Here the $A$-bimodule structure on $L$ is induced by the map $f$.
Note that if $(A,M,\vp)$ is an $s$-homogeneous triple and $n\ge s$ is an integer, then the map $1_{M^{\ot_A^i}}\ot_A\vp\ot_A1_{M^{\ot_A^{n-s-i}}}:M^{\ot_A^n}\rightarrow M^{\ot_A^{n-s}}[1]$ does not depend on $0\le i\le n-s$ due to the second item of Definition \ref{str}. For simplicity we denote this map by $\vp$ too. So, if $n\ge ks$, then we have a graded $A$-bimodule homomorphism $\vp^k:M^{\ot_A^n}\rightarrow M^{\ot_A^{n-ks}}[s]$. As usually, $M^{\ot_A^0}=A$ everywhere.

It was shown in \cite{MV} that the category of $s$-homogeneous triples is equivalent to the category of $s$-homogeneous algebras.
The corresponding equivalences $\FF$ and $\GG$ can be described on objects in the following way.
If $\Lambda$ is an $s$-homogeneous algebra, then $\FF(\Lambda)=(\Lambda^{(s)},\Lambda^{(s,1)},\phi^{(s)}_{\Lambda})$. 
Given an $s$-homogeneous triple $(A,M,\vp)$, one has $\GG(A,M,\vp)=T_\kk M_0/\big( (\Ker\vp)_0\big).$ Note also that $\GG(A,M,\vp)^!=T_\kk M_0^*/(\vp^*A_1^*)$.

The main ingredient of our approach, which we state in the following lemma, is \cite[Corollary 5.3]{MV}. 

\begin{lemma}\label{main_strip}
If $(A,M,\vp)$ is an $s$-homogeneous triple, then there exists a graded $A$-bimodule $S$ concentrated in degree $0$ and an isomorphism of graded $A$-bimodules $\theta:M^{\ot_A^s}\cong S\oplus J(A)[1]$ such that $\vp$ equals to the composition
$$
M^{\ot_A^s}\xrightarrow{\theta} S\oplus J(A)[1]\twoheadrightarrow J(A)[1]\hookrightarrow A[1],
$$
where the second map is the canonical projection on the second summand and the third arrow is the canonical inclusion.
\end{lemma}

In particular, it can be  deduced (see \cite[Corollary 5.4]{MV}) that if the $A$-bimodule $J(A)[1]$ does not contain nonzero direct summands concentrated in degree $0$, then the triple 
$(A,M,\vp)$ (and hence $\GG(A,M,\vp)$ too) can be recovered from the pair $(A,M)$ modulo isomorphism.
From here on and until the end of the paper, $(A,M,\vp)=\FF(\Lambda^!)$.

\subsection{$s$-Koszulity and $3$-Calabi-Yau property}\label{KozCY}

In this subsection we recall and discuss the definitions of an $s$-Koszul algebra, of a (twisted) $3$-Calabi-Yau algebra and of an algebra defined by a (twisted) potential.

\begin{Def} {\rm The $s$-homogeneous algebra $\Lambda$ is called {\it $s$-Koszul} if $\Ext_{\Lambda}^i(\kk,\kk)$ is concentrated in degree $-\chi_s(i)$, where
\begin{equation}
\chi_s(i)=\begin{cases}
\frac{is}{2},&\mbox{if $2\mid i$},\\
\frac{(i-1)s}{2}+1,&\mbox{if $2\nmid i$}.
\end{cases}
\end{equation}
The $2$-Koszul algebras are called simply {\it Koszul algebras}.
}
\end{Def}

If $s=2$ and $\Lambda$ is Koszul, then $\Ext_{\Lambda}^*(\kk,\kk)\cong (\Lambda^!)^{\rm op}$ as a graded algebra. For $s>2$ the situation changes.
It is proved in \cite{GMMZ} that if $\Lambda$ is $s$-Koszul, then $A$ is a Koszul algebra and $M$ is linear over $A$ and over $A^{\rm op}$.
If $\Lambda$ is $s$-Koszul and $s>2$, then $\Ext_{\Lambda}^*(\kk,\kk)\cong A\ltimes M$ as an algebra, where $A\ltimes M$ is the trivial extension of $A$ by $M$, i.e. its underlying space is $A\oplus M$ and the multiplication is given by the equality $(a,x)(b,y)=(ab,ay+xb)$ for $a,b\in A$ and $x,y\in M$. If we define the grading on $A\ltimes M$ by the equalities $(A\ltimes M)_{2n}=A_n$ and $(A\ltimes M)_{2n+1}=M_n$ for $n\ge 0$,{}  then the isomorphism above will become degree preserving.

 Let, as usually, $\Lambda$ be an $s$-homogeneous algebra and $\FF(\Lambda^!)=(A,M,\vp)$. Let us introduce the projective $\Lambda$-modules $K_{2n}=A_n^*\ot\Lambda$ and $K_{2n+1}=M_n^*\ot\Lambda$ for $n\ge 0$.
Note that the map $A_n\ot M_0\rightarrow M_n$ coming from the left $A$-module structure on $M$ induces the map $ M_n^*\rightarrow A_n^*\ot M_0^*$. Identifying of $M_0^*$ with $\Lambda_1$ and applying the multiplication $\Lambda_1\otimes\Lambda\rightarrow\Lambda$ we get the map $d_{2n}:K_{2n+1}\rightarrow K_{2n}$. Now we have also the map $\vp:M_n\ot M_0^{\ot (s-1)}\rightarrow A_{n+1}$ that induces the map $A_{n+1}^*\rightarrow M_n^*\ot (M_0^*)^{\ot (s-1)}$, and hence the map $d_{2n+1}:K_{2n+2}\rightarrow K_{2n+1}$. Then we get the complex of projective $\Lambda$-modules $\cdots\xrightarrow{d_2}K_2\xrightarrow{d_1}K_1\xrightarrow{d_0}K_0$ with the augmentation map $\mu:K_0=A_0\ot\Lambda\cong\Lambda\rightarrow\kk$ induced by the factorization by $J(\Lambda)$. The complex $(K,d)$ is called the {\it generalized Koszul complex} of $\Lambda$.
It was noted in \cite{MV} that our definition of the generalized Koszul complex is equivalent to the one given in \cite{B}.

It is proved in \cite{B} that $\Lambda$ is $s$-Koszul if and only if $K$ is exact in positive degrees. Let us recall that the Hilbert series of a graded space $U=\oplus_{i\ge 0}U_i$ is $\HH_U(t)=\sum\limits_{i\ge 0}(\dim_\kk U_i)t^i\in\bZ[[t]]$. Calculating the Euler characteristic of $K$ one gets immediately that
\begin{equation}\label{HSm}
\Big(\HH_A(t^s)-t\HH_M(t^s)\Big)\HH_{\Lambda}(t)=1.
\end{equation}
in the case where $\Lambda$ is $s$-Koszul. The opposite is not true in general even in the case $s=2$.

It is shown in \cite{B} that the exactness of $K$ implies the so-called {\it extra condition}
\begin{equation}\label{ec}
\begin{aligned}
    (\vp^*A_1^*)\ot (M_0^*)^{\ot(s-1)}\cap M_0^*\ot &\sum\limits_{i=0}^{s-2}\big((M_0^*)^{\ot i}\ot (\vp^*A_1^*)\ot (M_0^*)^{\ot (s-2-i)}\big)\\
    &=\big((\vp^*A_1^*)\ot M_0^*\cap M_0^*\otimes (\vp^*A_1^*)\big)\ot (M_0^*)^{\ot(s-2)},
\end{aligned}
\end{equation}
where both sides of the equality are subspaces of $(M_0^*)^{\ot(2s-1)}$.

In fact, this condition holds if and only if $(\Ker d_1)_i\subset \Im d_2$ for $i<2s$.
Let us introduce the $s$-homogeneous triple $(B,L,\psi)=\FF(\Lambda)$.
We denote by $O(t^n)$ the ideal of $\bZ[[t]]$ generated by $t^n$.
It follows from \cite[Corollary 7.4]{MV} that $\Lambda$ satisfies \eqref{ec} if and only if
\begin{equation}\label{HS}
\big(1-mt+(m^s-l_0)t^s-(m^{s+1}+l_1-2ml_0)t^{s+1}\big)\sum\limits_{k=0}^{s-1}(m^kt^k+l_kt^{k+s})-1\in O(t^{2s}),
\end{equation}
where $m=\dim_\kk L_0$ and $l_k=\dim_\kk (L^{\ot_Ak})_1$ for $0\le k\le s-1$.

It is proved in \cite{MV} that if $I$ is generated by one polynomial $f$, then the following result holds.

\begin{theorem}\label{MV_theor}
Suppose that $\Lambda=\kk\langle x_1,\dots,x_m\rangle/(f)$, where $f$ is some homogeneous polynomial  of degree $s$. Then $\Lambda$ is $s$-Koszul if and only if one of the following two conditions holds:
\begin{enumerate}
\item  $f=g^s$ for some linear polynomial $g$;
\item if $f=gh_1=h_2g$ for some polynomials $g$, $h_1$ and $h_2$, then $\deg g\in\{0,\deg f\}$.
\end{enumerate}
Moreover, if $s> 2$, then $\Lambda^!$ is $s$-Koszul if and only if $m=1$.
\end{theorem}

In the last theorem the case $f=g^s$ corresponds to the case $A=k[x]$.
If $f\not=g^s$ for any linear polynomial $g$, then $A=k[x]/(x^2)$ and $M=\kk^m$. In particular, the $s$-Koszulity in the last mentioned case is equivalent to the equality $\gd\Lambda=2$.
Moreover, if $\Lambda$ is not $s$-Koszul, then it does not satisfy the extra condition. The last assertion is a consequence of \cite{Gera} that is mentioned in \cite{B2}.
In fact, all of these facts can be obtained as consequences of the results of \cite{Dic}.

In this paper we will consider the case where $\Lambda=\kk\langle x_1,\dots,x_m\rangle/(f_1,f_2)$ for some linearly independent homogeneous polynomials $f_1$ and $f_2$ of degree $s$. It is shown in \cite{MV} that in this case $A=\kk\langle x,y\rangle/H$ for some quadratic ideal $H$ such that $\dim_\kk H_2\ge 2$. The consideration of the cases $H=(x^2, y^2)$ and $H=(xy, yx)$ is presented in \cite[Theorem 8.4]{MV}. Let us recall this result.

\begin{theorem}\label{A4A5} Suppose that $\Lambda=T_\kk V/(f_1, f_2)$, where $f_1$ and $f_2$ are two linearly independent elements of $V^{\ot s}$.\\
1. If $(\Lambda^!)^{(s)}\cong\kk\langle x,y\rangle/(xy,yx)$, then either
$$\Lambda\cong\kk\langle x_1,\dots,x_m,y_1,y_2\rangle/(y_1^s,y_2^s)$$
for some $m\ge 0$ or $s=2t$ and
$$\Lambda\cong\kk\langle x_1,\dots,x_m,y_1,y_2\rangle/\big((y_1y_2)^t,(y_2y_1)^t\big)$$
for some $t\ge 1$ and $m\ge 0$.\\
2. If $(\Lambda^!)^{(s)}\cong\kk\langle x,y\rangle/(x^2,y^2)$, then $s=2t+1$ and
$$\Lambda\cong\kk\langle x_1,\dots,x_m,y_1,y_2\rangle/\big((y_1y_2)^ty_1,y_2(y_1y_2)^t\big)$$
for some $t\ge 1$ and $m\ge 0$.
\end{theorem}
All the algebras $\Lambda$ mentioned in this theorem  are $s$-Koszul. Their $s$-homogeneous duals are not $s$-Koszul and it is a particular case of Theorem \ref{nkoz} presented below.

Let us recall that, for a $\Lambda$-bimodule $M$ and an automorphism $\theta$ of $\Lambda$, ${}_{\theta}M$ denotes the $\Lambda$-bimodule that coincides with $M$ as a right $\Lambda$-module and has the  left $\Lambda$-module structure twisted by $\theta$, i.e. defined by the equality $a\cdot x=\theta(a)x$ for $a\in\Lambda$ and $x\in M$.

\begin{Def}{\rm Let $\nu$ be an automorphism of an algebra $\Lambda$. Then the algebra $\Lambda$ is called {\it $\nu$-twisted $d$-Calabi-Yau} if the projective dimension of $\Lambda$ as a $\Lambda$-bimodule equals $d$, $\Ext_{\Lambda\otimes\Lambda^{\rm op}}^i(\Lambda,\Lambda\otimes\Lambda^{\rm op})=0$ for $i\not=d$, and $\Ext_{\Lambda\otimes\Lambda^{\rm op}}^d(\Lambda,\Lambda\otimes\Lambda^{\rm op})\cong {}_{\nu}\Lambda$ as a $\Lambda$-bimodule.
We will write CY instead of Calabi-Yau for short. We will call $1_{\Lambda}$-twisted $d$-CY algebras simply $d$-CY algebras. The automorphism $\nu$ is defined modulo the group of inner automorphisms by $\Lambda$. It  is called the {\it Nakayama automorphism} of $\Lambda$
}
\end{Def}

This definition makes sense for any algebra $\Lambda$ and any automorphism $\nu$, but in this paper we will use it only for graded (and even $s$-homogeneous) algebras and graded automorphisms.
Moreover, we will consider only the $3$-CY property. In this case algebras defined by twisted potentials play a significant role. Let us recall the definition of these algebras and their connections with CY algebras.
In fact, we will define next a homogeneous twisted potential, but we omit the term homogeneous because we are not going to consider non homogeneous potentials in this paper.

\begin{Def}{\rm 
Let $V$ be a finite dimensional vector space and $\sigma$ be some automorphism of $V$. An element $w\in V^{\ot N}$ is called {\it $\sigma$-twisted potential of degree $N$} if $\phi_\sigma(w)=w$, where $\phi_\sigma:V^{\ot N}\rightarrow V^{\ot N}$
acts on the tensor $v_1\ot\dots\ot v_N\in V^{\ot N}$ by the rule $\phi_\sigma(v_1\ot\dots\ot v_N)=\sigma(v_N)\ot v_1\ot\dots\ot v_{N-1}$. As before, we call $1_V$-twisted potentials simply potentials. The {\it algebra defined by the $\sigma$-twisted potential $w$} is by definition the algebra $\bD(V,w)=T_\kk V/I$, where $I$ is the ideal generated by the subspace $R=\{\big(f\ot 1^{\ot(N-1)}\big)(w)\mid f\in V^*\}\subset V^{\ot(N-1)}$. We will write simply $\bD(w)$ instead of $\bD(V,w)$, because $V$ can be always recovered from the context. It is clear from the definition that $\bD(w)$ is an $(N-1)$-homogeneous algebra. Note also that $\sigma:V\rightarrow V$ induces an automorphism of the algebra $\bD(w)$.
}
\end{Def}

\begin{rema}
Sometimes (for example, in \cite{Ginz}) a homogeneous potential is defined as a homogeneous element of the space $T_\kk V/[T_\kk V,T_\kk V]$. This definition gives the same definition of an algebra defined by a potential in the case of zero characteristic. In the case of an arbitrary characteristic considered here, our definition is more appropriate.
\end{rema}

It is shown in \cite{BSW} that any graded $\nu$-twisted $3$-CY algebra is $s$-Koszul for some $s$ and is isomorphic to an algebra defined by some $\nu$-twisted potential of degree $s+1$.
Here $\nu$ denotes the automorphism $\nu:\Lambda\rightarrow\Lambda$ and  at the same time the induced automorphism of $\Lambda_1$.

Let us recall one more definition.

\begin{Def}{\rm Let $\Lambda=\kk\langle x_1,\dots,x_m\rangle/I$ and $\Gamma=\kk\langle y_1,\dots,y_k\rangle/H$ be two $s$-homogeneous connected algebras. Their {\it free product} is by definition the algebra
$$\Lambda*\Gamma=\kk\langle x_1,\dots,x_m,y_1,\dots,y_k\rangle/(I+H),$$ where $I$ and $H$ are included in $\kk\langle x_1,\dots,x_m,y_1,\dots,y_k\rangle$ in the obvious way.
}
\end{Def}

\begin{Ex} Let us consider the potentials $w_1=y_1^{s+1}+y_2^{s+1}$ and $w_2=(y_1y_2)^{t+1}+(y_2y_1)^{t+1}$ on $\kk\langle y_1,y_2\rangle_{s+1}$ (where $s=2t+1$ in the case of the second potential).
It is easy to see that
\begin{multline*}
\kk\langle x_1,\dots,x_m\rangle*\bD(w_1)=\kk\langle x_1,\dots,x_m,y_1,y_2\rangle/(y_1^s,y_2^s),\\
\kk\langle x_1,\dots,x_m\rangle*\bD(w_2)=\kk\langle x_1,\dots,x_m,y_1,y_2\rangle/\big((y_1y_2)^ty_1,y_2(y_1y_2)^t\big)
\end{multline*}
are the algebras from Theorem \ref{A4A5}. Both of them are $s$-Koszul, but obviously not $3$-CY.
At the same time, $w_3=(y_1y_2)^{t}y_1+(y_2y_1)^{t}y_2\in \kk\langle y_1,y_2\rangle_{s+1}$ (where $s=2t$) is a $\nu$-twisted potential for $\nu$ interchanges $y_1$ and $y_2$. Then $\kk\langle x_1,\dots,x_m\rangle*\bD(w_3)=\kk\langle x_1,\dots,x_m,y_1,y_2\rangle/\big((y_1y_2)^t,(y_2y_1)^t\big)$ is the last algebra from Theorem \ref{A4A5} which is $s$-Koszul, but not $\nu$-twisted $3$-CY. Further we will meet (twisted) $3$-CY algebras too (see Subsection \ref{CY}).
\end{Ex}

\subsection{Classification of quadratic algebras with two generators}

We have not found an explicit classification of quadratic algebras with two generators modulo isomorphism. By this reason we present here this classification with a proof for completeness and  convenience.

\begin{lemma}\label{class} Let $A=\kk\langle x,y\rangle/H$, where $H$ is generated by $H_2$.\\
$1.$ If $\dim_\kk H_2=0$, then $A=\kk\langle x,y\rangle$ is the free algebra on two generators.

\noindent $2.$ If $\dim_\kk H_2=1$, then $A$ is isomorphic to one of the algebras
\begin{itemize}
\item  $\bar A^1=\kk\langle x,y\rangle/(x^2)$;
\item $ \bar A^2(q)=\kk\langle x,y\rangle/(xy+qyx)\,\,(q\in\kk)$;
\item $\bar A^3=\kk\langle x,y\rangle/(x^2+xy-yx)$.
\end{itemize}
We have the isomorphism $\bar A^2(q)\cong \bar A^2(q^{-1})$ for $q\not=0,1,-1$, all the other  algebras from the list are pairwise non-isomorphic.

\noindent $3.$ If $\dim_\kk H_2=2$, then $A$ is isomorphic to one of the algebras
\begin {itemize}
\item $ A^4=\kk\langle x,y\rangle/(xy,yx)$;
\item $ A^5=\kk\langle x,y\rangle/(x^2,y^2)$;
\item $A^6=\kk\langle x,y\rangle/(x^2,yx)$;
\item $A^7(q)=\kk\langle x,y\rangle/(x^2,xy-qyx)\,(q\in\kk)$;
\item $A^8=\kk\langle x,y\rangle/(xy,x^2-y^2)$;
\item $A^9(q)=\kk\langle x,y\rangle/(xy,x^2-yx-qy^2)\,(q\in\kk).$
\end{itemize}
All of these algebras are pairwise non-isomorphic.

\noindent $4.$ If $\dim_\kk H_2=3$, then $A$ is isomorphic to one of the algebras
\begin{itemize}
\item $A^1=\kk\langle x,y\rangle/(x^2,xy,yx)$;
\item $A^2(q)=\kk\langle x,y\rangle/(x^2,y^2,xy-qyx)\,\,(q\in\kk)$;
\item $A^3=\kk\langle x,y\rangle/(x^2,xy+yx,xy+y^2).$
\end{itemize}
Estanislao Herscovich We have the isomorphism $A^2(q)\cong A^2(q^{-1})$ for $q\not=0,1,-1$, all the other  algebras from the list are pairwise non isomorphic.

\noindent $5.$ If $\dim_\kk H_2=4$, then $A=\kk\langle x,y\rangle/(x,y)^2=A^0$.
\end{lemma}
\begin{proof} Items 1 and 5 are clear. Note that item 4 follows from item  2 because $(A^1)^!\cong \bar A^1$, $A^2(q)^!\cong \bar A^2(q)$, and $(A^3)^!\cong \bar A^3$.

Let us prove item 2. In this case $A=\kk\langle x,y\rangle/(ax^2+bxy+cyx+dy^2)$ for some nonzero vector $(a,b,c,d)\in \kk^4$.
Note that, for any $t\in\kk$, there is an isomorphism 
$$A\cong\kk\langle x,y\rangle/\big(ax^2+(b+at)xy+(c+at)yx+(at^2+(b+c)t+d)y^2\big)$$
sending $x$ to $x+ty$ and $y$ to $y$. Since $\kk$ is algebraically closed, we may assume that $A=\kk\langle x,y\rangle/(ax^2+bxy+cyx)$.
If $b+c\not=0$, then $A\cong\kk\langle x,y\rangle/(bxy+cyx)$ via the isomorphism sending $x$ to $x$ and $y$ to $y-\frac{a}{b+c}x$.
Now it is easy to turn $A$ into one of the required forms by rescaling and interchanging (in the case $b=0$, $c\not=0$) $x$ and $y$.

Direct calculations show that among the algebras $\bar A^1$, $\bar A^2(q)$ and $\bar A^3$ only the algebra $\bar A^1$ has square zero element of degree $1$, only the algebra $\bar A^2(0)$ has two linearly independent elements of degree $1$ whose product is zero, and, for $q\in\kk^*$, only the algebras $\bar A^2(q)$ and $\bar A^2(q^{-1})$ have such two linearly independent elements $u$ and $v$ of degree $1$ that $uv+qvu=0$.

It remains to prove item 3. We have
$A=\kk\langle x,y\rangle/(r_1,r_2)$, where 
$$r_1=a_1x^2+b_1xy+c_1yx+d_1y^2\mbox{ and }r_2=a_2x^2+b_2xy+c_2yx+d_2y^2$$ for some linearly independent vectors $(a_1,b_1,c_1,d_1),(a_2,b_2,c_2,d_2)\in \kk^4$. One can show that $t_1r_1+t_2r_2$ can be decomposed into a product of two linear polynomials for $t_1,t_2\in\kk$ if and only if the vectors
$$(a_1t_1+a_2t_2)x+(b_1t_1+b_2t_2)y\mbox{ and }(c_1t_1+c_2t_2)x+(d_1t_1+d_2t_2)y$$
are linearly dependent. Since this condition is equivalent to the vanishing of some homogeneous quadratic polynomial in $t_1$ and $t_2$,  we may assume that either $r_1=x^2$ or $r_1=xy$.

Suppose that $r_1=x^2$. Then we may assume that $r_2=bxy+cyx+dy^2$ for some $b,c,d\in\kk$. If $d=0$, then $A$ can be put in the form of $A^6$ or $A^7(q)$ by rescaling. In the opposite case we may assume that $d=1$.
If we also have $b=c$, then $A\cong A^5$ via the isomorphism sending $x$ to $x$ and $y$ to $y+bx$. If $b\not=c$, then $(y+bx)(y+cx)$ vanishes in $A$ and we reduce  to the case $r_1=xy$.

Suppose now that $r_1=xy$. We may assume that $r_2=ax^2+cyx+dy^2$ for some $a,c,d\in\kk$. If $a=d=0$, then $A=A^4$.
If $a=c=0$, then $A\cong A^6$ via interchanging $x$ and $y$. If $a=0$ and $c,d\not=0$, then changing $x$ by $y$ and $y$ by $cx+dy$ we may assume that $a\not=0$.
Now it is easy to turn $A$ into one of the required forms by rescaling $x$ and $y$.

Let us prove that all the algebras from item  3 are pairwise nonisomorphic. First of all, $A^4$ is the unique algebra from the list that is commutative and does not have square zero elements of degree $1$.  Only the algebra $A^5$ has two linearly independent square zero elements of degree $1$. Direct calculations show that the algebras $A^8$ and $A^9(q)$ do not have square zero elements while the algebras $A^6$ and $A^7(q)$ have unique modulo scalar square zero element of degree $1$. Thus, any graded isomorphism from $A^7(q)$ to $B\in\{A^7(q')\}_{q'\in\kk}\cup\{A^6\}$ has to send $x$ to $x$ and $y$ to $kx+ly$ for some $l\not=0$. One can check that this map is an isomorphism if and only if $B=A^7(q)$.
Finally, suppose that there is an isomorphism from $B\in\{A^9(q')\}_{q'\in\kk}\cup\{A^8\}$ to $A^9(q)$. Using the fact that $xy=0$ in $B$, one can show that this isomorphism has to send $x$ and $y$ either to $k_1x$ and $k_2y$ or to $k_1(x-y)$ and $k_2(x+qy)$ correspondingly for some $k_1,k_2\in\kk^*$. One can show that $k_1=k_2$ and $B=A^9(q)$ in the first case and that $q\not=-1$, $k_1+qk_2=0$, and $B=A^9(q)$ in the second case.
\end{proof}

It is not difficult to show, using the results of \cite{B0}, that the algebras $A^i$ ($1\le i\le 7$) are Koszul, we also may prove this by noting that the defining ideal for these algebras have quadratic Gr\"obner basis, which implies that they are Koszul. Algebras whose generating ideal have quadratic Gr\"obner basis are called strongly Koszul, see \cite{GW}. We also note  that the algebras $A^8$ and $A^9$ are not Koszul.

\subsection{Indecomposable modules}\label{ind_sec}

Let $A$ be a quadratic algebra with two generators, i.e. $A=\kk\langle x,y\rangle/H$, where $H$ is a quadratic ideal.
In this case an $A$-module $M$ with a linear presentation is determined by a pair $(U,T)$, where $U$ is some vector space and $T$ is a subspace of $U\otimes A_1$.
Namely, the pair $(U,T)$ corresponds to the cokernel of the map $T\ot A\rightarrow U\ot A$ induced by the inclusion of $T$ into $U\otimes A_1$.
 For instance, if $M= \kk = A/(x,y)$, then 
$U=\kk $ and $T= \kk\ot (\kk x \oplus \kk y).$ 
Two modules defined by $(U_1, T_1)$ and $(U_2, T_2)$ are isomorphic if and only if there is an isomorphism $f:U_1\rightarrow U_2$ such that $(f\otimes 1)(T_1)=T_2$.

Note that we do not use the relations of $A$ in this discussion. This means that  the category of linear presented modules does not depend on the defining ideal of the algebra.  Therefore the  modules having linear presentation over $A$ correspond bijectively in a natural way to the modules having linear presentation over  $A^0$.

Note that $A^0$ is a string special biserial algebra and the classification of indecomposable modules over it can be found, for example, in \cite{BR} (see also \cite{MV}). 

 Using the classification of indecomposable modules over the string special biserial algebra $ A^0$, we see that   any indecomposable $A$-module is isomorphic either to a module of the form
$$A^n/\langle af_1y,\{f_ix-f_{i+1}y\}_{1\le i\le n-1}, bf_nx\rangle_A\,\,\,(n\ge 1, a,b\in\{0,1\})$$
or to a module of the form
$$A^n/\langle\{(\alpha f_i+f_{i+1})x-f_{i}y\}_{1\le i\le n-1},\alpha f_nx-f_ny\rangle_A\,\,\,(n\ge 1, \alpha\in\kk^*).$$

 For $(\alpha,\beta)\in\kk^2\setminus\{(0,0)\}$ we choose some $(\tilde\alpha,\tilde\beta)\in\kk^2$ such that $\alpha\tilde\beta\not=\beta\tilde\alpha$ and introduce
$$
B_n(\alpha,\beta)=A^n/\langle \{f_{i+1}(\tilde\alpha x-\tilde\beta y)-f_{i}(\alpha x-\beta y)\}_{1\le i\le n-1},f_n(\alpha x-\beta y)\rangle_A\,\,\,(n\ge 1).
$$
It is not difficult to show that, for any nonzero $(\alpha,\beta)$, the module $B_n(\alpha,\beta)$ does not depend on the choice of $(\tilde\alpha,\tilde\beta)$ modulo isomorphism and is isomorphic to one of the modules listed above. Moreover, $B_n(\alpha,\beta)\cong B_m(\gamma,\delta)$ if and only if $n=m$ and $(\alpha,\beta)=(c\gamma,c\delta)$ for some $c\in\kk^*$. Thus, the classification of indecomposable $A$-modules can be rewritten in the following way. Any indecomposable $A$-module is isomorphic either to $B_n(\alpha,\beta)$ for some $(\alpha,\beta)\in\mathbb{P}^1$ or to one of the modules
$$
Z_n=A^n/\langle f_1y,\{f_ix-f_{i+1}y\}_{1\le i\le n-1}, f_nx\rangle_A,\,\,W_n=A^n/\langle\{f_ix-f_{i+1}y\}_{1\le i\le n-1}\rangle_A\,\,\,(n\ge 1).
$$
There is analogous classification for left modules. Namely, we introduce
\begin{multline*}
B'_n(\alpha,\beta)=A^n/\langle \{(\tilde\alpha x-\tilde\beta y)f_{i+1}-(\alpha x-\beta y)f_{i}\}_{1\le i\le n-1},(\alpha x-\beta y)f_n\rangle_{A^{\rm op}},\\
Z'_n=A^n/\langle yf_1,\{xf_i-yf_{i+1}\}_{1\le i\le n-1}, xf_n\rangle_{A^{\rm op}},\,\,W'_n=A^n/\langle\{xf_i-yf_{i+1}\}_{1\le i\le n-1}\rangle_{A^{\rm op}},
\end{multline*}
and then any indecomposable $A^{\rm op}$-module is isomorphic either to $B'_n(\alpha,\beta)$ for some $(\alpha,\beta)\in\mathbb{P}^1$ or to one of the modules
$Z'_n$, $W'_n$ ($n\ge 1$).  We remark that, given an $A$-module  $M$ isomorphic to  $B(\alpha, \beta)$ and
$(\alpha', \beta' )\in\mathbb{P}^1$, there is an isomorphism of algebras $A'\to A$, such that $M_{A'}$ is isomorphic to $B(\alpha', \beta').$

Now we give left and right module decompositions of $J(A)$, where $A$ is one of the algebras $A^i$ ($0\le i\le 9$) from Lemma \ref{class}. We collect this information in the next table.

\begin{center}
Table. {\it Decompositions of $J(A)$ for algebras with two generators.}\\
\begin{tabular}{|l|l|l|}
\hline
$A$&$J(A)[1]_A$&${}_AJ(A)[1]$\\
\hline
$A^0$&$Z_1\oplus Z_1$ &$Z_1'\oplus Z_1'$ \\
\hline
\hline
$A^1$&$B_1(1,0)\oplus Z_1$ &$B_1'(1,0)\oplus Z_1'$ \\
\hline
$A^2(0)$&$B_1(0,1)\oplus Z_1$ &$B_1'(1,0)\oplus Z_1'$ \\
\hline
$A^2(q)$ ($q\not=0$), $A^3$& $Z_2$ & $Z_2'$ \\
\hline
$A^4$, $A^5$&$B_1(1,0)\oplus B_1(0,1)$ & $B_1'(1,0)\oplus B_1'(0,1)$\\
\hline
$A^6$& $B_1(1,0)\oplus B_1(1,0)$&$W_1'\oplus Z_1'$ \\
\hline
$A^7(0)$&$W_1\oplus Z_1$ &$B_1'(1,0)\oplus B_1'(1,0)$ \\
\hline
$A^7(q)$ ($q\not=0$)&$B_2(1,0)$ &$B_2'(1,0)$ \\
\hline
$A^8$&$B_2(0,1)$ &$B_2'(1,0)$ \\
\hline
$A^9(q)$&$B_1(1,-q)\oplus B_1(0,1)$ & $B_1'(1,0)\oplus B_1'(1,1)$\\
\hline
\end{tabular}
\end{center}

\subsection{General ideas}
Our general goal is to classify the $s$-homogeneous algebras whose defining ideal have two generators, and to distinguish which ones are $s$-Koszul.

Since we have proved that the category of $s$-homogenous algebras is equivalent to the category of $s$-homogenous triples by a specific equivalence, we  will try to classify the corresponding $s$-homogeneous triples.

An $s$-homogenous triple is formed by a quadratic algebra, a graded bimodule and a morphism.

We observe that to  to classify $s$-homogeneous algebras with two relations it is enough to classify $s$-homogeneous triples $(A,M,\vp)$ whose first component is a quadratic algebra with two generators.
Thus, the first step of our classification has been done above.
The second step of our approach is the classification of  $A$-bimodules $M$ such that the pair $(A,M)$ can be completed to an $s$-homogeneous triple.
Lemma \ref{main_strip} giving restrictions on such a bimodule is the main ingredient of our second step. 

The third step is the classification  graded $A$-bimodule homomorphisms $\vp:M^{\ot_As}\rightarrow A[1]$ such that  $(A,M,\vp)$ is an $s$-homogeneous triple. 

The nice side of this  is that, for many algebras, if such a homomorphism exists, then, due to the remark after Lemma \ref{main_strip}, it is uniquely determined by the pair $(A,M)$, modulo isomorphism of $s$-homogeneous triples. In particular, $\vp$ is uniquely determined by $A$ and $M$ for all the algebras from Lemma \ref{class} except $A^1$ and $A^0$. 

The not so nice side  is that if $\vp$ exists and is not uniquely determined, then usually there is no hope to classify all the maps $\vp$. For example,
it is not done even for the algebra $A=k[x]/(x^2)$! Thus, for $A=A^1$ and $A=A^0$ we will not get a full classification but only some facts about corresponding $s$-homogeneous algebras.

Note that due to \cite[Theorem 2]{MSV} if $M$ and $N$ are graded $A$-bimodules linearly presented as left and right $A$-modules, then $M\ot_A N$ has the same property.
Thus, the category $\Lin(A)$ of graded $A$-bimodules linearly presented as left and right $A$-modules is a monoidal category. Note that $\kk\in\Lin(A)$ and if $A$ is a quadratic algebra, then $J(A)[1]\in\Lin(A)$ too. Note that $\kk\ot _AM\cong M\ot _A\kk\cong \kk^{\dim_\kk M_0}$ for any $M\in\Lin(A)$. Thus, morphisms that can be factored though $\kk^m$ for some $m\ge 0$ constitute a tensor ideal $\TT(A)$ in the category $\Lin(A)$ and we can consider the tensor category $\Lin_+(A)=\Lin(A)/\TT(A)$.

Due to Lemma \ref{main_strip}, if $(A,M,\vp)$ is an $s$-homogeneous triple, then $M\in\Lin(A)$ satisfies the condition $M^{\ot_As}\cong J(A)[1]$ in $\Lin_+(A)$.
Note that $M$ satisfies this condition if and only if $M\oplus\kk^m$ satisfies it for any $m\ge 0$. Moreover, it is clear that if $(A,M,\vp)$ is an $s$-homogeneous triple, then $(A,M\oplus \kk^m,\vp\pi^{\ot_As})$ is an $s$-homogeneous triple, where $\pi:M\oplus \kk^m\rightarrow M$ is the projection on the first summand.

\section{First two components}

Until the end of the paper we suppose that $A$ is a quadratic algebra with two generators, i.e. $A=\kk\langle x,y\rangle/H$, where $H$ is a quadratic ideal.
This section is devoted to the description of bimodules $M$ that can appear as the second component of an $s$-homogeneous triple with the first component $A$.
In particular, this will allow us to describe the $s$-homogeneous algebras corresponding to good algebras $A$, i.e. such that $J(A)[1]$ does not contain direct $A$-bimodule summands concentrated in degree zero.
Also, using the obtained description, we will show that a connected $s$-homogeneous algebra with two dimensional $s$-th component cannot be $s$-Koszul if $s>2$.

For an $A$-bimodule $N$, a subspace $X\subset N$, and a an element $a\in A$, we introduce
$$r\Ann_X(a)=\{w\in X\mid aw=0\}\mbox{ and }l\Ann_X(a)=\{w\in X\mid wa=0\}.$$
For $u\in A_1$ and $1\le t\le s$, let us introduce $r_u(t)=\dim_\kk\left(M_0^{\ot t}u\right)$ and $l_u(t)=\dim_\kk\left(uM_0^{\ot t}\right)$, where $M_0^{\ot t}$ is considered as a subspace of $M^{\ot_A t}$.

\subsection{First restrictions on $M$}

Here we prove the first lemma that restricts the number of left and right $A$-module structures that the bimodule $M$ from an $s$-homogeneous triple $(A,M,\vp)$ can have.

\begin{lemma}\label{main_rest} Suppose that $1\le t\le s$ and $(A,M,\vp)$ is an $s$-homogeneous triple, where $A=\kk\langle x,y\rangle/H$ for some quadratic ideal $H$.
\begin{enumerate}
\item There are linearly independent elements $u,v\in A_1$ such that $r_u(t),r_v(t)\le 2$.
\item If there is some $Y\in M_0^{\ot t}$ such that $\dim_\kk A_1Y=1$, then there exist  linearly independent $u,v\in A_1$ such that $r_u(t)\le 1$ and $r_v(t)\le 2$.
\item If the set of elements $Y\in M_0^{\ot t}$ such that $\dim_\kk A_1Y\le 1$ generate $M_0^{\ot t}$, then there are  linearly independent elements $u,v\in A_1$ such that $r_u(t),r_v(t)\le 1$.
\end{enumerate}
\end{lemma}
\begin{proof} Suppose that $u\in A_1$ can be represented in the form $u=\vp(X\ot Y)$ with $X\in M_0^{\ot (s-t)}$ and $Y\in M_0^{\ot t}$. Let us set $k=\dim_\kk A_1Y$.
Suppose that there is a set  $\{Z_i\}_{1\le i\le k+1}\subset M_0^{\ot t}$ such that the set $\{Z_iu\}_{1\le i\le k+1}\subset \left(M^{\ot_A t}\right)_1$ is linearly independent.
If there are $\alpha_i\in\kk$ ($1\le i\le k+1$) not all zero such that $\left(\sum\limits_{i=1}^{k+1}\alpha_i\vp(Z_i\ot X)\right)Y=0$, then we have
$$
\sum\limits_{i=1}^{k+1}\alpha_iZ_iu=\sum\limits_{i=1}^{k+1}\alpha_iZ_i\vp(X\ot Y)=\sum\limits_{i=1}^{k+1}\alpha_i\vp(Z_i\ot X)Y=0,
$$
which is a contradiction since $\{Z_iu\}_{1\le i\le k+1}$ is linearly independent. The obtained contradiction shows that $r_u(t)\le \dim_\kk A_1Y$.

\begin{enumerate}
\item Since the map $\vp|_{M_0^{\ot s}}:M_0^{\ot s}\rightarrow A_1$ is surjective, there are such $X_u,X_v\in M_0^{\ot (s-t)}$ and $Y_u,Y_v\in M_0^{\ot t}$ that $u=\vp(X_u\ot Y_u)$ and $v=\vp(X_v\ot Y_v)$ generate $A_1$. Then we have $r_u(t)\le \dim_\kk A_1Y_u\le 2$ and $r_v(t)\le \dim_\kk A_1Y_v\le 2$.
\item Since $\vp(M_0^{\ot s})Y=A_1Y\not=0$, there are such  $Z\in M_0^{\ot t}$ and $X\in M_0^{\ot (s-t)}$ that $Z\vp(X\ot Y)=\vp(Z\ot X)Y\not=0$, and hence $\vp(X\ot Y)\not=0$.
Then we may choose $Y_u=Y$ in the previous part of our proof and get $r_u(t)\le \dim_\kk A_1Y=1$.
\item In this case $M_0^{\ot s}$ is generated by elements of the form $X\ot Y$ with $X\in M_0^{\ot (s-t)}$, $Y\in M_0^{\ot t}$, and $\dim_\kk A_1Y\le 1$. Thus, we can choose $Y_u$ and $Y_v$ above in such a way that $\dim_\kk A_1Y_u\le 1$ and $\dim_\kk A_1Y_v\le 1$.
\end{enumerate}
\end{proof}

Note that we can apply Lemma \ref{main_rest} to the $s$-homogeneous triple $(A^{\rm op},M,\vp)$ and obtain the analogous assertion with $r_u(t)$ replaced by $l_u(t)$. The same argument works for all assertions that we will prove (with appropriate modifications). Thus,
whenever we prove some statement for right modules, we will propose the analogous fact for left modules.
Most of the time we will apply Lemma \ref{main_rest} with $t=1$. Let us give an immediate corollary of it.

\begin{coro}\label{first_coro}
 In the settings of Lemma \ref{main_rest}, $\dim_\kk \left(M^{\ot_A t}\right)_1\le 4$ and $\left(M^{\ot_A t}\right)_A$ does not contain
any direct summands isomorphic to $B_n(\alpha,\beta)$ $((\alpha,\beta)\in\mathbb{P}^1)$, $Z_{n+1}$ and $W_n$ with $n\ge 3$.
 \end{coro}

 We are going to prove a series of claims that improve the assertion of Corollary \ref{first_coro}, but first we need to introduce some notation.
 
The notation $N\ds_A K$ means that  $N$ is an $A$-module direct summand of $K$. We will write $N\nds_A K$ for the opposite assertion.

We claim that if $W_1 \ds_A K\otimes_A N$ for an $A$-module $K$ and an $A$-bimodule $N$, then $W_1 \ds_A N$.
To see the claim first observe that  $W_1\cong A_A$ and  choose some $A$-module epimorphism $\pi:W_1^n\rightarrow K$. Then there is an epimorphism $\pi\ot_A1_N:N^n\rightarrow K\ot_AN$. Hence, if $W_1\ds_AK\ot_AN$, then there is an $A$-module epimorphism from $N^n$ to $W_1$. Since $W_1$ is projective, $W_1\ds_AN^n$, and hence $W_1\ds_AN$.
 This shows the claim.
 
 \begin{lemma}\label{W1} In the settings of Lemma \ref{main_rest}, we have $W_1\ds_AM^{\ot_At}$ if and only if $A\cong A^7(0)$. In this case $M_A\cong W_1\oplus Z_1^l$ for some $l\ge 1$.
 \end{lemma}
 \begin{proof} By the argument above, we have $W_1\ds_AM$ if $W_1\ds_AM^{\ot_Ak}$ for some $k\ge 1$. In particular, if $A\cong A^7(0)$, then $W_1\ds_AJ(A)[1]\ds_AM^{\ot_A s}$, and hence $W_1\ds_AM$. 
 
  Suppose now that $M_A\cong W_1\oplus K$ for some $A$-module $K$. Let us prove by induction that $M\ds_AM^{\ot_A k}$  for any $k\ge 1$. This statement is obviously true for $k=1$. For $k>1$ we have the isomorphism
 $$\left(M^{\ot_A k}\right)_A\cong (W_1\oplus K)\ot_AM^{\ot_A (k-1)}\cong \left(M^{\ot_A (k-1)}\right)_A\oplus \left(K\ot_AM^{\ot_A (k-1)}\right),$$
 and hence $M\ds_AM^{\ot_A k}$ follows from $M\ds_AM^{\ot_A (k-1)}$.
 In particular, $W_1\ds_AM^{\ot_A s}\cong J(A)[1]\oplus S$ with $S$ concentrated in degree zero by Lemma \ref{main_strip}, and hence $A\cong A^7(0)$ due to Table of decompositions presented in Subsection \ref{ind_sec}.
 Suppose that $A\cong A^7(0)$ and $M\cong W_1\oplus K$. Since $M\ds_AM^{\ot_A s}$ in this case, we have
 $$2+\dim_\kk K_1=\dim_\kk M_1\le\dim_\kk (M^{\ot_A s})_1=\dim_\kk J(A)_2=2.$$
 Thus, in this case $K$ is concentrated in degree zero.
 \end{proof}

 As usually, the statement dual to Lemma \ref{W1} is true too. Note that in such a statement $A^7(0)$ has to be replaced by $A^6$.
We will continue the search of restrictions on the direct summands of $M$ and its tensor powers, but due to Lemma \ref{W1} it is convenient to exclude the algebras $A^6$ and $A^7(0)$ first.
 
\subsection{$s$-homogeneous triples with $A=A^6 \text{ or } A = A^7(0)$}

In this subsection  we will classify the $s$-homogeneous algebras $\Lambda$ with two relations such that $\FF(\Lambda^!)=(A,M,\vp)$ with $A\in\{A^6,A^7(0)\}$. Note that $A^7(0)=(A^6)^{\rm op}$, and hence, if $\FF(\Lambda^!)=(A^7(0),M,\vp)$, then $\FF\big((\Lambda^{\rm op})^!\big)=(A^6,M^{\rm op},\vp^{\rm op})$. Thus, one can consider the case $A=A^6$ and then apply ${}^{\rm op}$ to get the classification for $A=A^7(0)$. So we set $A=A^6$ in our discussion.

Let us consider the $s$-homogeneous triple $(A,M,\vp)$. We have ${}_AM=W_1'\oplus K'$ for some $K'$ concentrated in degree zero by Lemma \ref{W1}. Then $\dim_\kk M_1=2$, and so the right $A$-module structure of $M$ can be recovered using the following lemma.

\begin{lemma}\label{zero_x} Suppose that $(A^6,M,\vp)$ is an $s$-homogeneous triple. Then $M_0x=0$.
\end{lemma}
\begin{proof} Note that we have a direct decomposition $M_1=xM_0\oplus yM_0$, where $\dim_\kk xM_0=\dim_\kk yM_0=1$. We have $A_1(xM_0)=0$ and $A_1(yM_0)=M_2$, and hence, if $M_0x\not\subset xM_0$, then $M_2=A_1M_0x=M_1x$.
Let us denote by $\pi:M^{\ot s}\twoheadrightarrow M^{\ot_A s}$ the canonical projection.
Then we have $(M^{\ot_A s})_2=\pi\big(M_0^{\ot (s-1)}\ot M_2\big)=\pi\big(M_0^{\ot (s-1)}\ot M_1x\big)\subset (M^{\ot_A s})_1x=0$ that is not true.

Thus, $M_0x\subset xM_0$, i.e. $M_0x=xM_0$ in the case $M_0x\not=0$. On the other hand, if $M_0x=xM_0$, then $x M_0^{\ot s}=M_0^{\ot s}x$ which is not true. Thus, $M_0x=0$.
\end{proof}

Lemma \ref{zero_x} implies that $M_A\cong B_1(1,0)\oplus  B_1(1,0)\oplus K$ for some $K$ concentrated in degree zero. Now we are ready to describe $s$-homogeneous algebras with two relations corresponding to $A^6$ and $A^7(0)$.

\begin{theorem}\label{A6A70} Suppose that $\Lambda=T_\kk V/(f_1, f_2)$, where $f_1$ and $f_2$ are two linearly independent elements of $V^{\ot s}$.\\
1. If $(\Lambda^!)^{(s)}\cong\kk\langle x,y\rangle/(x^2,yx)=A^6$, then
$$\Lambda\cong\kk\langle x_1,\dots,x_m,y_1,y_2\rangle/(y_1^s,y_2y_1^{s-1})$$
for some $m\ge 0$.\\
2. If $(\Lambda^!)^{(s)}\cong\kk\langle x,y\rangle/(x^2,xy)=A^7(0)$, then
$$\Lambda\cong\kk\langle x_1,\dots,x_m,y_1,y_2\rangle/(y_1^s,y_1^{s-1}y_2)$$
for some $m\ge 0$.

In particular, if $(\Lambda^!)^{(s)}\in\{A^6,A^7(0)\}$, then $\Lambda$ is $s$-Koszul.
\end{theorem}
\begin{proof} As it was explained above, we will consider only the case where $\FF(\Lambda^!)\cong (A^6,M,\vp)$ and then will get the result for $A^7(0)$ via the functor ${}^{\rm op}$. Let us set $A=A^6$ and fix the $A$-module decomposition $M=B_1(1,0)\oplus  B_1(1,0)\oplus K$ with $A$-module $K$ concentrated in degree zero. Our first step is to recover the $A$-bimodule structure on $M$.

First of all, note that $A_1K\subset l\Ann_{M_1}(y)=0$. Then $r\Ann_{M_0}(x)=r\Ann_{M_0}(y)$ is a subspace of $M_0$ of codimension one containing $K$. Let $f_1,f_2$ be such a basis of $\big(B_1(1,0)\oplus  B_1(1,0)\big)_0\subset M_0$ that $A_1f_2=0$. i.e. $f_2$ is a complementary vector of $K$ in $r\Ann_{M_0}(A_1)$ and $f_1$ is a complementary vector of $\langle f_2\rangle+K$ in $M_0$. Note that $f_1y^i,f_2y^i$ is a basis of $M_i$ for any $i\ge 1$. Since $0\not=xf_1\in M_1$ and $x(xf_1)=0$, it is easy to see that $xf_1=af_2y$ for some $a\in\kk^*$. Now we may rescale $f_2$ and turn $a$ into one.
Thus, we have $xf_1=f_2y$ and $yf_1=af_1y+bf_2y$ for some $a,b\in\kk$, $a\not=0$. Let us replace $f_1$ by $f_1+\frac{b}{a}f_2$. Now we have $xf_1=f_2y$ and $yf_1=af_1y$. Now one can show that tensors of the form $v_1\ot\dots\ot v_s\in M_0^{\ot s}$ such that either $v_i\in K$ for some $1\le i\le s$ or $v_i=f_2$ for some $2\le i\le s$ generate an $A$-bimodule concentrated in degree zero. The subspace generated by these tensors has codimension two in $M_0^{\ot s}$ and its complementary subspace is $\langle f_1^{\ot s},f_2\ot f_1^{\ot (s-1)}\rangle$. Since $A_1(f_2\ot f_1^{\ot (s-1)})=0$, we have $\vp(f_2\ot f_1^{\ot (s-1)})=bx$ and $\vp(f_1^{\ot s})=cx+dy$ for some $b,d\in\kk^*$, $c\in\kk$. Since $df_1y=f_1\vp(f_1^{\ot s})=\vp(f_1^{\ot s})f_1=adf_1y+cf_2y$, we have $a=1$ and $c=0$. It is not difficult to show that we get an $s$-homogeneous triple $(A^6,J(A^6)[1]\oplus \kk^m,\vp)$ with $m=\dim_\kk K$ that does not depend on $\vp$ modulo isomorphism because $J(A^6)[1]$ does not contain direct $A^6$-bimodule summands concentrated in degree zero.
Using the description of $(\Ker\vp)_0$ we can conclude that $\Lambda\cong \GG(A^6,J(A^6)[1]\oplus \kk^m,\vp)^!=T\big((J(A^6)_1\oplus \kk^m)^*\big)/\big(\vp^*(A^6)_1^*\big)$ has the required form.
\end{proof}

\subsection{More restrictions on the one side module structure}

Let us continue the study of $s$-homogeneous triples. Since we have fully considered $s$-homogeneous triples corresponding to the algebras $A^6$ and $A^7(0)$, we will proceed our discussion in the following settings:
\begin{equation}\label{settings}
\begin{array}{c}
(A,M,\vp)\mbox{ is an $s$-homogeneous triple},\\
A=\kk\langle x,y\rangle/H, \mbox{ for some quadratic ideal $H$, with } A\not\cong A^6,\mbox{ and } A\not\cong A^7(0)
\end{array}
\end{equation}

 \begin{lemma}\label{Z}
 In the settings \eqref{settings}, $N\nds_AM$ for $N\in\{Z_3,Z_2\oplus Z_2\}$.
 \end{lemma}
 \begin{proof} Suppose that $M_A\cong N\oplus K$ for some $A$-module $K$ and $N\in\{Z_3,Z_2\oplus Z_2\}$. For any nonzero $w\in A_1$, we have
$$
r_w(1)=\dim_\kk (N_0w)+\dim_\kk (K_0w)=2+\dim_\kk (K_0w)\ge 2.
$$
Then it follows from Lemma \ref{main_rest} that $\dim_\kk A_1Y\not=1$ for any
 $Y\in M_0$ and $K_0u=K_0v=0$ for two linearly independent $u,v\in A_1$, i.e. $K_1=K_0A_1=0$.
 The first condition means that indecomposable direct summands of ${}_AM$ are isomorphic to $Z_1'$ and $W_n'$ ($n\ge 1$). Since $\dim_\kk M_1=\dim_\kk N_1=2$, we have ${}_AM\cong W_1'\oplus K'$ for some $K'$ concentrated in degree zero that contradicts Lemma \ref{W1}. Thus, the required assertion is proved.
 \end{proof}

 \begin{lemma}\label{W2} In the settings \eqref{settings}, $W_2\nds_AM$.
 \end{lemma}
 \begin{proof} Suppose that $M_A\cong W_2\oplus K$ for some $A$-module $K$.  Analogously to the proof of Lemma \ref{Z}, one can show that $K$ is concentrated in degree zero and indecomposable direct summands of ${}_AM$ are isomorphic to $Z_1'$ and $W_n'$ ($n\ge 1$). 
  Since $\dim_\kk M_1=\dim_\kk (W_2)_1$, we have ${}_AM\cong W_2'\oplus K'$ for some $K'$ concentrated in degree zero. A direct calculation shows that $\dim_\kk (W_2\ot_AW_2')_1=4$, and hence $\dim_\kk (M\ot_AM)_1=4$ too.
Thus, it follows from Lemma \ref{main_rest} that $\dim_\kk fA_1\not=1$ for any $f\in (M\ot_AM)_0$. Then indecomposable direct summands of $(M\ot_AM)_A$ are isomorphic to $Z_1$ and $W_n$ ($n\ge 1$). Since $\dim_\kk (M\ot_AM)_1=4$, we have either $W_1\oplus W_1\ds_AM\ot_AM$ that contradicts Lemma \ref{W1} or $W_3\ds_AM\ot_AM$ that contradicts Corollary \ref{first_coro}. Thus, $W_2\nds_AM$.
 \end{proof}

 \begin{lemma}\label{B2} In the settings \eqref{settings}, if $M_A\cong B_2(\alpha,\beta)\oplus K$  for some $(\alpha,\beta)\in\mathbb{P}^1$ and $A$-module $K$, then $K$ is concentrated in degree zero and $B_2'(\alpha',\beta')\ds_{A^{\rm op}}M$ for some $(\alpha',\beta')\in\mathbb{P}^1$.
 \end{lemma}
 \begin{proof} Suppose that $M_A\cong B_2(\alpha,\beta)\oplus K$ for some $A$-module $K$. Let us first consider the case where $\dim_\kk A_1f\not=1$ for any
 $f\in M_0$. In this case indecomposable direct summands of ${}_AM$ are isomorphic to $Z_1'$ and $W_n'$ ($n\ge 1$). Which is impossible by Corollary \ref{first_coro} and Lemmas \ref{W1}, and \ref{W2}.
 
 Suppose now that $\dim_\kk A_1f=1$ for some $f\in M_0$. Then by Lemma \ref{main_rest} there are $u,v\in A_1$ such that $r_u(1)\le 1$ and $r_v(1)\le 2$.
It is easy to see that $\dim_\kk \big(B_2(\alpha,\beta)_0(\alpha x-\beta y)\big)=1$ and $\dim_\kk \big(B_2(\alpha,\beta)_0w\big)=2$ for any $w\in A_1$ linearly independent with $\alpha x-\beta y$. Since
$r_w(1)=\dim_\kk \big(B_2(\alpha,\beta)_0w\big)+\dim_\kk(K_0w)$
for any $w\in A_1$, we have $r_{\alpha x-\beta y}(1)\ge 1$ and $r_w(1)\ge 2$ for  any $w\in A_1$ linearly independent with $\alpha x-\beta y$. Thus, we may assume that $u=\alpha x-\beta y$ and $K_0u=K_0v=0$, i.e. $K_1=K_0A_1=0$ and $K$ is concentrated in degree zero. Moreover, if $B_2'(\alpha',\beta')\nds_{A^{\rm op}}M$ for any $(\alpha',\beta')\in\mathbb{P}^1$, then ${}_AM$ is a direct sum of modules of the form $Z_1'$, $Z_2'$, and $B_1'(\alpha',\beta')$ ($(\alpha',\beta')\in\mathbb{P}^1$). It is easy to see that in this case $M_0$ is generated by such elements $f\in M_0$ that $\dim_\kk A_1f\le 1$ that contradicts to the fact that $r_w(1)\ge 2$ for  any $w\in A_1$ linearly independent with $\alpha x-\beta y$.
 \end{proof}
 
  \begin{lemma}\label{B1} In the settings of Lemma \ref{main_rest}, $N\nds_AM$ for
  $$N\in\{B_1(\alpha,\beta)\oplus Z_2,B_1(\alpha,\beta)\oplus B_1(\alpha,\beta)\}_{(\alpha,\beta)\in\mathbb{P}^1}.$$
 \end{lemma}
 \begin{proof} Suppose that $N\ds_AM$, where $N\cong B_1(\alpha,\beta)\oplus Z_2$ or $N\cong B_1(\alpha,\beta)\oplus B_1(\alpha,\beta)$ for some $(\alpha,\beta)\in\mathbb{P}^1$. Then our standard argument (see, for example, the proof of Lemma \ref{B2}) shows that $r_w(1)\ge 2$ for any $w$ linearly independent with $\alpha x-\beta y$. Then $M_0$ cannot be generated by elements $f\in M_0$ such that $\dim_\kk A_1f\le 1$ due to item 3 of Lemma \ref{main_rest}. Thus, ${}_AM$ has to contain a direct summand isomorphic to one of the modules $W_n$, $Z_{n+2}$ or $B_{n+1}(\alpha',\beta')$ ($n\ge 1$, $(\alpha',\beta')\in\mathbb{P}^1$). Corollary \ref{first_coro} and Lemmas \ref{W1}, \ref{Z}, \ref{W2}, and \ref{B2} imply that we get a contradiction, and hence $N\nds_AM$.
 \end{proof}
 
 \begin{coro}\label{second_coro} In the settings \eqref{settings}, $M_A\cong N\oplus K$, where $K$ is concentrated in degree zero and $N$ belongs to the set
 $$
\{0,Z_2,B_1(\alpha,\beta),B_2(\alpha,\beta),B_1(\alpha,\beta)\oplus B_1(\gamma,\delta)\}_{(\alpha,\beta),(\gamma,\delta)\in\mathbb{P}^1,(\alpha,\beta)\not=(\gamma,\delta)}.
 $$
 In particular, $\dim_\kk M_1\le 2$.
 \end{coro}
 \begin{proof} Let us represent $M_A$ in the form $M_A\cong N\oplus K$, as in the statement,  with $N$ without direct summands concentrated in degree zero.
We have  $W_i\nds_AM$ for any $i\ge 1$ by Corollary \ref{first_coro} and Lemmas \ref{W1} and \ref{W2}.
 If $B_2(\alpha,\beta)\ds_AM$ for some $(\alpha,\beta)\in\mathbb{P}^1$, then the required assertion follows from Lemma \ref{B2}.
 In the remaining case it follows from Corollary \ref{first_coro} and Lemmas \ref{Z} and \ref{B1} that if $N\not\cong Z_2$, then indecomposable direct summands of $N$ are isomorphic to $B_1(\alpha,\beta)$ ($(\alpha,\beta)\in\mathbb{P}^1$) and that all of them are pairwise nonisomorphic.
 It remains to show that $\dim_\kk M_1\le 2$ in this case. This follows from item 3 of Lemma \ref{main_rest} and the fact that elements $f\in M_0$ with $\dim_\kk fA_1\le 1$ generate $M_0$.
 \end{proof}

\subsection{Combining one side structures}

Corollary \ref{second_coro} gives  strong restrictions on possible one side $A$-module structures of $M$. Now we are going to combine possible right structures with possible left structures. Here, for each right $A$-module structure on $M$, we will determine the possible left $A$-module structures of $M$ and then the left and right $A$-module structures on $M^{\ot_Ak}$ for any $k>1$.  Thus, in view of our Table of decompositions for $J(A)$, we will be able to recover the algebra $A$ from the left and right $A$-module structures of $M$. In particular, we will exclude some of the algebras from the list of candidates for the first position of an $s$-homogeneous triple.

Let us first prove a lemma that will be useful in our discussion.

\begin{lemma}\label{ker_vp}
Suppose that $M_0u\not=0$ for any nonzero $u\in A_1$ in the settings \eqref{settings}. Then $\vp\big(M_0^{\ot (k-1)}\ot r\Ann_{M_0}(A_1)\ot M_0^{\ot (s-k)}\big)=0$ for any $1\le k\le s$.
\end{lemma}
\begin{proof} We will prove the required assertion using descending induction by $k$. Suppose that $\vp(X\ot f)=u\not=0$ for some $X\in M_0^{\ot (s-1)}$ and $f\in r\Ann_{M_0}(A_1)$. Then we have $gu=g\vp(X\ot f)=\vp(g\ot X)f=0$ for any $g\in M_0$, i.e. $M_0u=0$. The obtained contradiction shows that the required assertion is valid for $k=s$.

Suppose that $1\le k<s$ and we have already proved that
$$\vp\big(M_0^{\ot k}\ot r\Ann_{M_0}(A_1)\ot M_0^{\ot (s-k-1)}\big)=0.$$
Suppose that $\vp(X\ot f\ot Y\ot h)=u\not=0$ for some $X\in M_0^{\ot (k-1)}$, $f\in r\Ann_{M_0}(A_1)$, $Y\in M_0^{\ot (s-k-1)}$, and $h\in M_0$. Then we have
$$gu=g\vp(X\ot f\ot Y\ot h)=\vp(g\ot X\ot f\ot Y)h=0$$
for any $g\in M_0$ by the induction hypothesis, i.e. $M_0u=0$. The obtained contradiction shows that the required assertion is valid for any $1\le k\le s$.
\end{proof}

\begin{lemma}\label{B2_ts} Suppose that in the settings \eqref{settings} one has $M_A\cong B_2(\alpha,\beta)\oplus Z_1^m$ for some $(\alpha,\beta)\in\mathbb{P}^1$ and $m\ge 0$. Then $(M^{\ot_Ak})_A\cong B_2(\alpha,\beta)\oplus Z_1^{(m+2)^k-2}$ and ${}_A(M^{\ot_Ak})\cong B_2'(\alpha,\beta)\oplus (Z_1')^{(m+2)^k-2}$ for any $k\ge 1$.
\end{lemma}
\begin{proof} We have ${}_AM\cong B_2'(\alpha',\beta')\oplus (Z_1')^m$ for some $(\alpha',\beta')\in\mathbb{P}^1$ by Lemma \ref{B2}. Note that $\vp\left(M_0^{\ot (i-1)}\ot (Z_1')^m\ot M_0^{\ot (s-i)}\right)=\vp\left(M_0^{\ot (i-1)}\ot (Z_1)^m\ot M_0^{\ot (s-i)}\right)=0$ for any $1\le i\le s$ by Lemma \ref{ker_vp}. If $Z_1^m\not=(Z_1')^m$, then the space $Z_1^m+(Z_1')^m\subset M_0$ has codimension less or equal to one, and hence $2=\dim_\kk\Im\vp|_{M_0^{\ot s}}\le 1$. The obtained contradiction shows that $Z_1^m=(Z_1')^m$, and hence $M\cong N\oplus \kk^m$, where $N_A\cong B_2(\alpha,\beta)$ and ${}_AN\cong B_2'(\alpha',\beta')$. Moreover, $\vp$ factors through the canonical projection $M^{\ot_As}\twoheadrightarrow N^{\ot_As}$.

Suppose first that $(\alpha',\beta')\not=(\alpha,\beta)$. Note that after a replacement of $A$ by an isomorphic algebra we may assume that $(\alpha,\beta)=(1,0)$ and $(\alpha',\beta')=(0,1)$.

Let us denote by $f_1$ and $f_2$ the standard generators of $B_2(1,0)$ and by $f_1'$ and $f_2'$ the standard generators of $B_2(0,1)$.
Let us prove that $\vp\left(N_0^{\ot (i-1)}\ot f_2'\ot N_0^{\ot (s-i)}\right)\subset \langle x\rangle$ for any $1\le i\le s$ by descending induction on $i$. Suppose that $\vp(Y\ot f_2')=b_xx+b_yy$ with $b_y\not=0$ for some $Y\in N_0^{\ot (s-1)}$. Let us choose $c_1,c_2\in\kk$ not all zero such that $c_1\vp(f_1\ot Y)+c_2\vp(f_2\ot Y)\in\langle y\rangle$. Then we have
$0=\vp\big((c_1f_1+c_2f_2)\ot Y\big)f_2'=(c_1f_1+c_2f_2)\vp(Y\ot f_2')\not=0$. The obtained contradiction shows that $\vp\left(N_0^{\ot (s-1)}\ot f_2'\right)\subset \langle x\rangle$.
Suppose now that $\vp\left(N_0^{\ot i}\ot f_2'\ot N_0^{\ot (s-i-1)}\right)\subset \langle x\rangle$ for some $1\le i<s$. Let us take arbitrary $Y\in N_0^{\ot (i-1)}$, $Z\in N_0^{\ot (s-i-1)}$ and $h\in N_0$. Let us choose $c_1,c_2\in\kk$ not all zero such that $c_1\vp(f_1\ot Y\ot f_2'\ot Z)+c_2\vp(f_2\ot Y\ot f_2'\ot Z)=0$. It is possible because $\vp(f_1\ot Y\ot f_2'\ot Z),\vp(f_2\ot Y\ot f_2'\ot Z)\in\langle x\rangle$ by induction hypothesis. We have
$$(c_1f_1+c_2f_2)\vp(Y\ot f_2'\ot Z\ot h)=\big(c_1\vp(f_1\ot Y\ot f_2'\ot Z)+c_2\vp(f_2\ot Y\ot f_2'\ot Z)\big)h=0,$$
and hence $\vp(Y\ot f_2'\ot Z\ot h)\in \langle x\rangle$. In analogous way, one can show that $\vp\left(N_0^{\ot (i-1)}\ot f_2\ot N_0^{\ot (s-i)}\right)\subset \langle y\rangle$ for any $1\le i\le s$.
If $f_2$ and $f_2'$ are linearly dependent, then $\vp\left(N_0^{\ot (i-1)}\ot f_2\ot N_0^{\ot (s-i)}\right)=0$ for any $1\le i\le s$, and hence we again get $2=\dim_\kk\Im\vp|_{M_0^{\ot s}}\le 1$. If $f_2$ and $f_2'$ are linearly independent, then
they constitute a basis of $N_0$, and hence we have $\vp(f_2^{\ot s})=cy$ for some $c\in\kk^*$. Then we have $0\not=\frac{1}{c}f_2'y=f_2'\vp(f_2^{\ot s})=\vp(f_2'f_2^{\ot (s-1)})f_2=0$. The obtained contradiction shows that the case $(\alpha',\beta')\not=(\alpha,\beta)$ is impossible.

Now we already have ${}_AM\cong B_2'(\alpha,\beta)\oplus (Z_1')^m$. We can assume again that $(\alpha,\beta)=(1,0)$. Direct calculations show that $\dim_\kk\big(B_2(1,0)\ot_AB_2'(1,0)\big)_1=2$. Note that, for any $u\in A_1$ linearly independent with $x$, $M_0u=M_1$, and hence $(M\ot_AM)_0u=(M\ot_AM)_1$ too. As usually, the analogous assertion holds for ${}_A(M\ot_AM)$.
It follows that that $r_u(2)=2$ for any $u\in A_1$ linearly independent with $x$, and hence $M_0\ot M_0$ is not generated by elements $f\in M_0\ot M_0$ such that $\dim_\kk(A_1f)\le 1$.
Since nonsemisimple direct summands of ${}_A(M\ot_AM)$ can be $B_2'(1,0)$, $B_1'(1,0)$ and $Z_2'$, the last assertion means that $B_2'(1,0)\ds_{A^{\rm op}}M\ot_AM$. Then the dual argument shows that $B_2(1,0)\ds_{A}M\ot_AM$. Thus, we have $(M\ot_AM)_A\cong B_2(1,0)\oplus Z_1^{(m+2)^2-2}$ and ${}_A(M\ot_AM)\cong B_2'(1,0)\oplus (Z_1')^{(m+2)^2-2}$.
In particular, this means that $B_2(1,0)\ot_AM\cong B_2(1,0)\oplus Z_1^{2m+2}$ and $M\ot_AB_2(1,0)'\cong B_2'(1,0)\oplus (Z_1')^{2m+2}$.
Now the required assertion follows by induction on $k$.
\end{proof}

\begin{lemma}\label{B11_ts} Suppose that in the settings \eqref{settings} one has ${}_AM\cong B_1(\alpha,\beta)\oplus B_1(\gamma,\delta)\oplus Z_1^m$ for some $(\alpha,\beta),(\gamma,\delta)\in\mathbb{P}^1$ and $m\ge 0$. Then $(M^{\ot_Ak})_A\cong B_1(\alpha,\beta)\oplus B_1(\gamma,\delta)\oplus Z_1^{(m+2)^k-2}$ and ${}_A(M^{\ot_Ak})\cong B_1'(\alpha,\beta)\oplus B_1'(\gamma,\delta)\oplus (Z_1')^{(m+2)^k-2}$ for any $k\ge 1$.
\end{lemma}
\begin{proof} We have ${}_AM\cong B_1'(\alpha',\beta')\oplus B_1'(\gamma',\delta')\oplus (Z_1')^m$ for some $(\alpha',\beta'),(\gamma',\delta')\in\mathbb{P}^1$ by Lemma \ref{B2} and Corollary \ref{second_coro}. Note that $(\alpha,\beta)\not=(\gamma,\delta)$ and $(\alpha',\beta')\not=(\gamma',\delta')$ by Lemma \ref{B1}.

Note that $M_0$ is generated by elements $f'\in M_0$ such that $\dim_\kk(A_1f')\le 1$ and by elements $f\in M_0$ such that $\dim_\kk(fA_1)\le 1$.
Let us choose among the elements of the form $\vp(f\ot X\ot f')\in A_1$ ($f,f'\in M_0$, $\dim_\kk(fA_1),\dim_\kk(A_1f')\le 1$, $X\in M_0^{\ot (s-2)}$) two linearly independent elements $u$ and $v$.
Following the proof of Lemma \ref{main_rest}, we obtain that $r_u(1)\le 1$, and hence $u\in\langle\alpha x-\beta y\rangle$ or $u\in\langle\gamma x-\delta y\rangle$. Analogous inclusion for $v$ and the dual argument show that 
$$\{\langle\alpha' x-\beta' y\rangle,\langle\gamma' x-\delta' y\rangle\}=\{\langle u\rangle,\langle v\rangle\}=\{\langle\alpha x-\beta y\rangle,\langle\gamma x-\delta y\rangle\},$$
and hence $\{(\alpha',\beta'),(\gamma',\delta')\}=\{(\alpha,\beta),(\gamma,\delta)\}$.

Now we already have ${}_AM\cong B_1'(\alpha,\beta)\oplus B_1'(\gamma,\delta)\oplus (Z_1')^m$. After a replacement of $A$ by an isomorphic algebra we may assume that $(\alpha,\beta)=(1,0)$ and $(\gamma,\delta)=(0,1)$. Now it is easy to see that
\begin{multline*}
\dim_\kk\big(B_1(1,0)\ot_AB_1'(1,0)\big)_1=\dim_\kk\big(B_1(0,1)\ot_AB_1'(0,1)\big)_1=1\mbox{ and}\\
\big(B_1(1,0)\ot_AB_1'(0,1)\big)_1=\big(B_1(0,1)\ot_AB_1'(1,0)\big)_1=0.
\end{multline*}
Thus, $\dim_\kk(B_1(1,0)\ot_AM)_1=\dim_\kk(B_1(0,1)\ot_AM)_1=1$, i.e. $(M\ot_A M)_A$ has two indecomposable direct summands with one dimensional first components. Since $M(ax+by)=M$ for any $a,b\in\kk^*$, these direct summands are isomorphic to $B_1(1,0)$, $B_1(0,1)$ and $Z_2$. Analogously ${}_A(M\ot_A M)$ has two indecomposable direct summands isomorphic to $B_1'(1,0)$, $B_1'(0,1)$ or $Z_2'$. Since $M_0\ot M_0$ is generated by such elements $f\in M_0\ot M_0$ that $\dim_\kk A_1f\le 1$, it is easy to check using Lemma \ref{main_rest} that only the case $B_1(1,0)\oplus B_1(0,1)\ds_AM\ot_AM$ is possible, i.e. $(M\ot_AM)_A\cong B_1(1,0)\oplus B_1(0,1)\oplus Z_1^{(m+2)^2-2}$.
In particular, this means that $\big(B_1(1,0)\oplus B_1(0,1)\big)\ot_AM\cong B_1(1,0)\oplus B_1(0,1)\oplus Z_1^{2m+2}$.
Now the required formula for $(M^{\ot_Ak})_A$ follows by induction on $k$. The proof of the formula for ${}_A(M^{\ot_Ak})$ is dual.
\end{proof}

\begin{lemma}\label{Z2_ts} Suppose that in the settings \eqref{settings} one has $M_A\cong Z_2\oplus Z_1^m$ for some $m\ge 0$. Then ${}_AM\cong Z_2'\oplus (Z_1')^m$ and $M^{\ot_Ak}\cong \kk^{(m+2)^k}$ for any $k>1$.
\end{lemma}
\begin{proof} First, suppose that ${}_AM\cong B_1'(\alpha,\beta)\oplus (Z_1')^{m+1}$ for some $(\alpha,\beta)\in\mathbb{P}^1$.
Let us set $K':=r\Ann_{M_0}(A_1)=(Z_1')^{m+1}$. Then we have $\vp(M_0^{\ot (k-1)}\ot K'\ot M_0^{\ot (s-k)})=0$ for any $1\le k\le s$ by Lemma \ref{ker_vp}.
Since $K'$ has codimension one in $M_0$, we get $2=\dim_\kk\Im\vp|_{M_0^{\ot s}}\le 1$.
The obtained contradiction shows that ${}_AM$ does not have a direct summand of the form $B_1'(\alpha,\beta)$, and hence ${}_AM\cong Z_2'\oplus (Z_1')^m$. It is easy to see that $(Z_2\ot_AZ_2')_1=0$, and thus $M^{\ot_Ak}$ is concentrated in degree zero for any $k>1$.
\end{proof}

\begin{lemma}\label{B1_ts} Suppose that in the settings \eqref{settings} one has $M_A\cong B_1(\alpha,\beta)\oplus Z_1^m$ for some $(\alpha,\beta)\in\mathbb{P}^1$ and $m\ge 0$. Then ${}_AM\cong  B_1(\alpha',\beta')\oplus (Z_1')^m$ for some $(\alpha',\beta')\in\mathbb{P}^1$. Moreover,
\begin{enumerate}
\item $M^{\ot_Ak}\cong \kk^{(m+1)^k}$ for any $k>1$ if $(\alpha',\beta')\not=(\alpha,\beta)$;
\item $(M^{\ot_Ak})_A\cong B_1(\alpha,\beta)\oplus Z_1^{(m+1)^k-1}$ and ${}_A(M^{\ot_Ak})\cong B_1'(\alpha,\beta)\oplus (Z_1')^{(m+1)^k-1}$ for any $k>1$ if $(\alpha',\beta')=(\alpha,\beta)$.
\end{enumerate}
\end{lemma}
\begin{proof} The fact that ${}_AM$ has the form stated in the lemma follows from Lemma \ref{Z2_ts}. Now the first part follows from the fact that $\big(B_1(\alpha,\beta)\ot_AB_1'(\alpha',\beta')\big)_1=0$ for  $(\alpha',\beta')\not=(\alpha,\beta)$.

Finally, let us consider the case $(\alpha',\beta')=(\alpha,\beta)$. Since
$\dim_\kk\big(B_1(\alpha,\beta)\ot_AB_1'(\alpha,\beta)\big)_1=1$, we have $\dim_\kk(M\ot_AM)_1=1$. Since $(M\ot_AM)_0(\alpha x-\beta y)=(\alpha x-\beta y)(M\ot_AM)_0=0$, it is easy to see that the required assertion is true for $k=2$. In particular, this means that $B_1(\alpha,\beta)\ot_AM\cong B_1(\alpha,\beta)\oplus Z_1^m$ and $M\ot_AB_1'(\alpha,\beta)\cong B_1'(\alpha,\beta)\oplus (Z_1')^m$. Then the required assertion follows  by induction on $k$.
\end{proof}

Now we can prove the next part of our classification.

\begin{theorem}\label{impos_struc} Suppose that $\Lambda=T_\kk V/(f_1, f_2)$, where $f_1$ and $f_2$ are two linearly independent elements of $V^{\ot s}$.
Then $(\Lambda^!)^{(s)}$ is not isomorphic to $A^2(q)$, $A^3$, $A^8$ and $A^9(q)$ for any $q\in\kk$.
\end{theorem}
\begin{proof} Suppose that the assertion of the theorem is not true. Then there is an $s$-homogeneous triple $(A,M,\vp)$ with $A\in \{A^2(q),A^3,A^8,A^9(q)\}_{q\in\kk}$.
One can use Corollary \ref{second_coro}, Lemmas \ref{B2_ts}, \ref{B11_ts}, \ref{Z2_ts}, and \ref{B1_ts}, and Table of decompositions for $J(A)$ to show that $M^{\ot_As}$ cannot have ${}_AJ(A)[1]$ as a left $A$-module direct summand and $J(A)[1]_A$ as a right $A$-module direct summand at the same time. This contradicts Lemma \ref{main_strip}, and thus finishes the proof of the theorem.
\end{proof}

\subsection{Non Koszulity of $\Lambda^!$}

Here, using the results about $s$-homogeneous triples $(A,M,\vp)$ with quadratic algebra $A$ having two generators and the formula \eqref{HS} for Hilbert series of $s$-Koszul algebras, we show that an algebra $\Gamma=T_\kk V/I$ with $I$ generated by a subspace $W\subset V^{\ot s}$ of codimension two cannot be $s$-Koszul if $s>2$.

\begin{theorem}\label{nkoz} Suppose that $s>2$ and $\Gamma=T_\kk V/(W)$, where $W$ is a subspace of $V^{\ot s}$. If $\dim_\kk\Gamma_s=2$, then $\Gamma$ is not $s$-Koszul.
\end{theorem}
\begin{proof} If $\Gamma$ is $s$-Koszul, then it satisfies the extra condition, and hence \eqref{HS} holds with $m=\dim_\kk M_0$ and $l_k=\dim_\kk (M^{\ot_Ak})_1$ for $0\le k\le s-1$.

Let us calculate the coefficient of $t^{s+2}$ in the formal power series on the left hand side of $\eqref{HS}$. Observe that $l_0= 2$, so we obtain that $\Gamma$ may be $s$-Koszul only if
$$
0=l_2-ml_1+m^2(m^s-2)-m(m^{s+1}+l_1-4m)=2m^2-2l_1m+l_2,
$$
where $l_1=\dim_\kk M_1$ and $l_2=\dim_\kk (M\ot_AM)_1$. We have $l_1\le 2$ by Lemma \ref{W1} and Corollary \ref{second_coro}. Moreover, it follows from the same results and Lemmas \ref{B2_ts} and \ref{B11_ts} that $l_2=2$ in the case $l_1=2$. Since  $m\ge 2$, we have $2m(m-l_1)+l_2>0$, and thus $\Gamma$ cannot be $s$-Koszul.
\end{proof}

\subsection{$s$-homogeneous triples with $A=A^7(q)$, $q\not=0$}\label{A7sec}

Theorem \ref{impos_struc} says that the cases $A=A^2(q)$, $A=A^3$, $A=A^8$, and $A=A^9(q)$ are impossible.  
The cases $A=A^4$ and $A=A^5$ are fully considered in Theorem \ref{A4A5} while the cases $A=A^6$ and $A=A^7(0)$ are fully considered in Theorem \ref{A6A70}. Thus, it remains to consider the cases $A=A^0$, $A=A^1$, and $A=A^7(q)$ ($q\not=0$).  Among these algebras only $A^7(q)$ give $s$-homogeneous algebras that can be recovered from the pair $(A,M)$. Because of this, the  case $A^7(q)$ is easier to deal with, and we consider it now.  This is the last case for which this paper gives a full classification. In the next section we will give some partial results on the classification in the cases $A=A^0$ and $A=A^1$.

Until the end of this section we fix some $q\in\kk^*$ and set $A=A^7(q)$. Let $(A,M,\vp)$ be an $s$-homogeneous triple.
Due to the previous results of the paper and the Table of decompositions for $J(A)$, we have $M_A\cong B_2(1,0)\oplus Z_1^m$ and ${}_AM\cong B_2'(1,0)\oplus (Z_1')^m$ for some $m\ge 0$.
Thus, we need to recover the $A$-bimodule structure gluing the given one side structures on $M$. Then the triple $(A,M,\vp)$ can be recovered uniquely modulo isomorphism.

Let us introduce the graded $A^7(q)$-bimodule $M(\lambda)$ for $\lambda\in\kk^*$ by the equality
$$
M(\lambda)=(A^{\rm op}\ot A)^2/\langle yf_1-f_1y, xf_1-\lambda f_1x, f_1x-f_2y, f_2x,xf_2, yf_2-q^{-1}f_2y \rangle_{A^{\rm op}\ot A}.
$$
As usually, $f_1$ and $f_2$ are the standard free generators. It is not difficult to show that $M(\lambda)$ is an $A^7(q)$-bimodule such that $M(\lambda)_A\cong B_2(1,0)$ and ${}_AM(\lambda)\cong B_2'(1,0)$. In particular, it has a linear presentation as a left and as a right module.

\begin{lemma}\label{iso7} Suppose that there is an automorphism $\theta$ of $A$ and an isomorphism of graded spaces $\gamma:M(\lambda)\cong M(\mu)$ such that $\gamma(ufv)=\theta(u)\gamma(f)\theta(v)$ for all $u,v\in A$, $f\in M(\lambda)$.
Then $\lambda=\mu$.
\end{lemma}
\begin{proof} For $a,b\in\kk^*$ let us consider the function $\kappa_{a,b}:\mathbb{P}^1\rightarrow \bZ$ defined by the equality
$$
\kappa_{a,b}(\alpha,\beta)=\dim_\kk\{f\in M(b)_0\mid (\alpha x-\beta y)f=af(\alpha x-\beta y)\}.
$$
It is easy to see that $\kappa_{a,\lambda}=\kappa_{a,\mu}\circ g$, where $g:\mathbb{P}^1\rightarrow \mathbb{P}^1$ is the linear transformation induced by $\theta|_{A_1}$. One can show now that $\kappa_{a,b}$ has nonzero values if and only if $a\in\{1,q^{-1},b\}$. Moreover, if $q\not=1$, then
$\kappa_{q^{-1},b}$ has value $2$ in some point if and only if $b=q^{-1}$.
Thus, $b$ can be recovered if we know the set of values of $\kappa_{a,b}$ for all $a\in\kk^*$. Since these sets coincide for all $a$ for the elements $\lambda$ and $\mu$, we have $\lambda=\mu$.
\end{proof}

\begin{lemma}\label{ten7} $M(\lambda)\ot_A M(\mu)\cong M(\lambda\mu)\oplus\kk^2$.
\end{lemma}
\begin{proof} Let us denote by $g_1,g_2$ the standard generators of $M(\lambda)$ and by $h_1,h_2$ the standard generators of $M(\mu)$. It is not difficult to check that
$$A_1(g_2\ot h_2)=(g_2\ot h_2)A_1=A_1(g_2\ot h_1-\mu g_1\ot h_2)=(g_2\ot h_1-\mu g_1\ot h_2)A_1=0,$$
i.e. there is a homomorphism from $(A^{\rm op}\ot A)^2\oplus\kk^2$ to $M(\lambda)\ot_A M(\mu)$ sending the standard generators $f_1$ and $f_2$ of $(A^{\rm op}\ot A)^2$ to $g_1\ot h_2$ and $g_1\ot h_1$ and the basis $e_1,e_2$ of $\kk^2$ to $g_2\ot h_2$ and $g_2\ot h_1-\mu g_1\ot h_2$ correspondingly. Direct calculations show that this homomorphism factors throw the natural projection $ (A^{\rm op}\ot A)^2\oplus\kk^2\twoheadrightarrow M_2(\lambda\mu)\oplus\kk^2$. The obtained map $\gamma:M(\lambda\mu)\oplus\kk^2\rightarrow M(\lambda)\ot_A M(\mu)$ is surjective. Since both bimodules are linearly presented and have the same dimensions of the first and the second components, $\gamma$ is an isomorphism.
\end{proof}

\begin{lemma}\label{shp7} If $(A^7(q),M,\vp)$  is an $s$-homogeneous triple, then $(A^7(q),M,\vp)\cong(A^7(q),M(p)\oplus\kk^m,\vp_p)$ for some $m\ge 0$, $p\in\kk^*$ such that $p^s=q$ and some map $\vp_p$.
\end{lemma}
\begin{proof} Due to the proof of Lemma \ref{B2_ts}, we may assume that $M=N\oplus \kk^m$, where $N$ is an $A$-bimodule isomorphic to $B_2(1,0)$ as an $A$-module. Let $f_1,f_2$ be the standard generators of $N$ as an $A$-module $B_2(1,0)$, i.e. $f_1,f_2\in N_0$, $f_1x=f_2y$ and $f_2x=0$. Note that the basis of $N_i$ is $f_1y^i,f_2y^i$ for any $i\ge 0$. Since $l\Ann_{N_1}(x)=\langle f_2y\rangle$, we have $A_1f_2\subset \langle f_2y\rangle$. We have also
$$r\Ann_{N_0}(x)=\{f\in N_0\mid \dim_\kk (A_1f)=1\},$$ i.e. $xf_2=0$ and $yf_2=af_2y$ for some $a\in \kk^*$ (note that $r\Ann_{N_0}(y)=0$). Moreover, $xN_0=r\Ann_{N_1}(x)$, and hence $xf_1=pf_2y$ for some $p\in\kk^*$. Finally, we have $yf_1=bf_1y+cf_2y$ for some $b,c\in\kk$, $b\not=0$, because $A_1N_0=N_1$.

We have $\vp(M^{\ot_Ai}\ot_A \kk^m\ot_AM^{\ot_A(s-1-i)})=0$ for any $1\le i\le s$ by Lemma \ref{ker_vp}. Moreover, it is easy to show using the equality $xf_1=pf_1x$ that $x(f_{i_1}\ot\dots\ot f_{i_s})=0$, and hence $\vp(f_{i_1}\ot\dots\ot f_{i_s})\in\langle x\rangle$ if one of the numbers $i_1,\dots,i_s\in\{1,2\}$ equals $2$.
Since $A_1\subset \Im\vp$, we have $\vp\left(f_1^{\ot s}\right)=\alpha x+\beta y$ for some $\alpha,\beta\in\kk$, $\beta\not=0$. Then we have
\begin{equation}\label{eq_0}
\beta f_1y+\alpha f_2y=f_1\vp\left(f_1^{\ot s}\right)=\vp\left(f_1^{\ot s}\right)f_1=\beta bf_1y+(\alpha p+\beta c)f_2y,
\end{equation}
and hence $b=1$. Now, using the equality $0=(xy-qyx)f_1=(p-qap)f_2y^2$, we get $a=q^{-1}$.

Suppose that $q\not=1$. The map sending $f_1$ to $f_1+\frac{qc}{q-1} f_1$ and $f_2$ to $f_2$ is an automorphism of $N_A$.
At the same time, we have $y\left(f_1+\frac{qc}{q-1} f_2\right)=\left(f_1+\frac{qc}{q-1} f_2\right)y$. Hence, we may assume $c=0$ in the case $q\not=1$.

Let us now consider the case $q=1$. If we have also $p=1$, then $c=0$ by \eqref{eq_0}. Suppose that $p\not=1$. Then we can change the $A$-bimodule structure of $M$ via the automorphism of $A^7(q)$ sending $x$ to $x$ and $y$ to $y+\frac{c}{1-p}x$. Note that in this way we will not change the right structure. Since $\left(y+\frac{c}{1-p}x\right)f_1=f_1\left(y+\frac{c}{1-p}x\right)$, we again may assume that $c=0$.

Thus, it remains to show that $p^s=q$. But we have $(M(p)\oplus\kk^m)^{\ot_As}\cong M(p^s)\oplus \kk^{(m+2)^s-2}$ by Lemma \ref{ten7}. Since $J\big(A^7(q)\big)\cong M(q)$, we have $p^s=q$ by Lemma \ref{iso7}.
\end{proof}

\begin{theorem}\label{A7q} Suppose that $\Lambda=T_\kk V/(f_1, f_2)$, where $f_1$ and $f_2$ are two linearly independent elements of $V^{\ot s}$.
Suppose that $(\Lambda^!)^{(s)}\cong\kk\langle x,y\rangle/(x^2,xy-qyx)=A^7(q)$ for some $q\in\kk^*$. Then
$$\Lambda\cong\kk\langle x_1,\dots,x_m,y_1,y_2\rangle/(y_1^s,\sum\limits_{i=0}^{s-1}p^{s-i-1}y_1^iy_2y_1^{s-i-1})$$
for some $m\ge 0$ and some $p\in\kk^*$ such that $p^s=q$. In particular, $\Lambda$ is $s$-Koszul.
\end{theorem}
\begin{proof} We may assume $\FF(\Lambda^!)=(A^7(q),M(p)\oplus\kk^m,\vp_p)$ by Lemma \ref{shp7}. It remains to describe the kernel of $\vp_p$.
Let $f_1$ and $f_2$ be the standard generators of $M(p)$.
 As we have mentioned above, $\vp_p$ can have nonzero values only on the tensors of the form $X=f_{i_1}\ot\dots\ot f_{i_s}$ with $i_1,\dots,i_s\in\{1,2\}$. Also if one of the indices $i_k$ ($1\le k\le s$) equals $2$, then $\vp_p(X)\in\langle x\rangle$. Let us prove that $\vp_p(X)=0$ if there are two different $1\le k<l\le s$ such that $i_k=i_l=2$. It is enough to prove that $A_1X=0$. We have already mentioned that $xX=0$. We may assume $i_1=i_2=\dots=i_{k-1}=1$. Then we have
 $yX=f_1^{\ot(k-1)}\ot yf_2\ot X'=\frac{1}{q}f_1^{\ot k}\ot xX'=0$, because the tensor $X'$ contains $f_2$ by our assumption. 
 Let us now look at the equality $f_1\vp_p(f_1^{\ot (i-1)}\ot f_2\ot f_1^{\ot(s-i)})=\vp_p(f_1^{\ot i}\ot f_2\ot f_1^{\ot(s-i-1)})f_2$ for $1\le i\le s-1$. Substituting
 $\vp_p(f_1^{\ot k}\ot f_2\ot f_1^{\ot(s-k-1)})=a_kx$ with $a_k\in\kk$ ($0\le k\le s-1$) to this equality, we get $a_i=\frac{1}{p}a_{i-1}$.
 Thus, we have proved that decomposable tensors with a component belonging to $\kk^m$, decomposable tensors with minimum two components equal to $f_2$ and the tensors $pf_1^{\ot i}\ot f_2\ot f_1^{\ot(s-i-1)}-f_1^{\ot (i-1)}\ot f_2\ot f_1^{\ot(s-i)}$ ($1\le i\le s-1$) belong to the kernel of $\vp_p$. Since we have obtained a subspace of codimension two, the listed elements generate $\Ker\vp_p$.
Now it is easy to see that  $\Lambda\cong \GG(A^7(q),M(p)\oplus\kk^m,\vp_p)^!=T\big((M(p)_0\oplus\kk^m)^*\big)/\big(\vp_p^*A^7(q)_1\big)$ has the required form.

The $s$-Koszulity of $\Lambda$ follows from the fact that $G=\{y_1^s,\sum\limits_{i=0}^{s-1}p^{s-i-1}y_1^iy_2y_1^{s-i-1}\}$ is a Gr\"obner basis with respect to the lexicographical order such that $y_1>y_2>x_1>\dots>x_m$. Indeed, we have $tip(G)=\{y_1^s,y_1^{s-1}y_2\}$ and the required assertion follows, for example, from \cite[Corollary 12]{MSV}.
\end{proof}

\section{Triples not determined by two components}

The only cases that have not been considered yet are $A=A^0$ and $A=A^1$ (see the beginning of Subsection \ref{A7sec}). The main difficulty of these cases is that $J(A)[1]$ contains simple $A$-bimodule summands, and thus $\Lambda$ cannot be recovered from the pair $(A,M)$ while the classification of maps $\vp:M^{\ot^{s}_A }\rightarrow A[1]$ turning $(A,M,\vp)$ into an $s$-homogeneous triple is at least as difficulty as the classification of all $s$-homogeneous algebras with one relation.  Nevertheless, in these cases we can say something about the algebra $\Lambda$.

It follows from Corollary \ref{second_coro}, Lemmas \ref{B2_ts}, \ref{B11_ts}, \ref{Z2_ts} and \ref{B1_ts} and Table of decompositions for $J(A)$ that all the pairs that we have to consider satisfy one of the conditions:
\begin{enumerate}
    \item $A=A^1$, $M_A\cong B_1(1,0)\oplus Z_1^m$, ${}_AM\cong B_1'(1,0)\oplus (Z_1')^m$;
    \item $A=A^0$, $M_A\cong Z_2 \oplus Z_1^m$, ${}_AM\cong Z_2'\oplus (Z_1')^m$
    \item $A=A^0$, $M_A\cong B_1(\alpha,\beta)\oplus Z_1^m$, ${}_AM\cong B_1'(\alpha',\beta')\oplus (Z_1')^m$ for some $(\alpha,\beta),(\alpha',\beta')\in \mathbb{P}^1$, $(\alpha,\beta)\not=(\alpha',\beta')$. In this case after a change of basis we may assume that $(\alpha,\beta)=(1,0)$ and $(\alpha',\beta')=(0,1)$;
    \item $A=A^0$, $M=\kk^m$.
\end{enumerate}

In this section we will give some addition information in each of these cases. 



\subsection{$s$-homogeneous triples with $A=A^1$}

Throughout this subsection we consider  $A:=A^1$. As it was mentioned above, we have $M_A\cong B_1(1,0)\oplus Z_1^m$ and ${}_AM\cong B_1'(1,0)\oplus (Z_1')^m$ in this case. Let us first finish the description of the $A$-bimodule $M$.
Note that the bimodule $A/(x)$ satisfies the conditions $\big(A/(x)\big)_A\cong B_1(1,0)$ and ${}_A\big(A/(x)\big)\cong B_1'(1,0)$. Moreover, $\big(A/(x)\big)^{\ot_A k}\cong A/(x)$ for any $k\ge 1$.

\begin{lemma}\label{bimod_B1B1}
There is an isomorphism of $A$-bimodules $M\cong A/(x)\oplus \kk^{m}$.
\end{lemma}
\begin{proof} We know that $M_A\cong B_1(1,0)\oplus Z_1^m$. Let us recover the left $A$-module structure on $B_1(1,0)\oplus Z_1^m$ corresponding to the left $A$-module structure of $M$. Note that $l\Ann_{M_1}(y)=0$ and $(y Z_1^m)y=0$, and hence $yZ_1^m=0$. Hence, we have an isomorphism of graded $A$-bimodules  $M\cong N\oplus\kk^m$, where $N$ is a graded $A$-bimodule such that $N_A\cong B_1(1,0)$ and ${}_AN\cong B_1'(1,0)$. Let $f\in N_0$ be a nonzero element. Then we have $yf=\alpha fy$ for some $\alpha\in\kk^*$ and all that we need is to prove that $\alpha=1$. Note that $(f_1\ot\dots\ot f_s)y=0$ for $f_1\ot\dots \ot f_s\in M_0^{\ot s}$ such that $f_i\in\kk^m$ for some $1\le i\le s$, and hence $\vp(f_1\ot\dots\ot f_s)\in \langle x\rangle$ in this case. Since $y\in\Im\vp$ we have $\vp(f^{\ot s})=\beta x+\gamma y$ for some $\beta\in\kk$ and $\gamma\in\kk^*$. Then we have
$$
fy=\frac{1}{\gamma}f\vp(f^{\ot s})=\frac{1}{\gamma}\vp(f^{\ot s})f=yf=\alpha fy,
$$
i.e. $\alpha=1$.
\end{proof}

\begin{theorem}\label{A1}
Suppose that $\Lambda=T_\kk V/(f_1, f_2)$, where $f_1$ and $f_2$ are two linearly independent elements of $V^{\ot s}$.

If $(\Lambda^!)^{(s)}\cong\kk\langle x,y\rangle/(x^2,xy,yx)=A^1$, then $\Lambda\cong\kk\langle x_1,\dots,x_m,y_1\rangle/(f,y_1^s)$
for some $f\in \kk\langle x_1,\dots,x_m,y_1\rangle_s$.
Conversely, if $\Lambda=\kk\langle x_1,\dots,x_m,y_1\rangle/(f,y_1^s)$, then $(\Lambda^!)^{(s)}\cong A^1$ if and only if there are no linear polynomial $g\in \kk\langle x_1,\dots,x_m,y_1\rangle$ and $\alpha,\beta\in\kk$ such that
either $f=\sum\limits_{i=0}^{s-1}\alpha^i\beta^{s-i-1}y_1^igy_1^{s-i-1}$ or $f=g^s+\alpha y_1^s$.
\end{theorem}
\begin{proof} We have $\Lambda\cong T_\kk M_0^*/(\vp^*A_1^*)$, where $(A,M,\vp)$ is an $s$-homogeneous triple. If $(\Lambda^!)^{(s)}\cong A^1$, then we may take $A=A^1$ and, by Lemma \ref{bimod_B1B1}, $M=A/(x)\oplus \kk^{m}$.

Let now $x_1,\dots,x_m,y_1\in M_0^*$ be the basis dual to a basis $f_1,\dots,f_{m+1}$ of $M_0$ whose first $m$ elements constitute a basis of the summand $\kk^{m}$ and the last one is some nonzero element of the summand $A/(x)$.
Let now $f_x,f_y$ be the basis of $A_1^*$ dual to the basis $x,y$ of $A_1$.
It follows from the proof of Lemma \ref{bimod_B1B1} that $f_y\vp(f_{i_1}\ot\dots\ot f_{i_s})=0$ if $i_k\not=m+1$ for some $1\le k\le s$, and hence $\vp^*f_y=\alpha y_1^s$ for some $\alpha\in\kk^*$. Thus, $\vp^*A_1^*=\langle f, y_1^s\rangle$ for some $f\in \kk\langle x_1,\dots,x_m,y_1\rangle_s$.

Let us now assume that $\Lambda=\kk\langle x_1,\dots,x_m,y_1\rangle/(f,y_1^s)$ for some $f\in \kk\langle x_1,\dots,x_m,y_1\rangle_s$ linearly independent with $y_1^s$. If $f=g^s+\alpha y_1^s$, then $\Lambda=\kk\langle x_1,\dots,x_m,y_1\rangle/(g^s,y_1^s)$, and hence $(\Lambda^!)^{(s)}\cong A^4$ (see Theorem \ref{A4A5}). If $f=\sum\limits_{i=0}^{s-1}\alpha^i\beta^{s-i-1}y_1^igy_1^{s-i-1}$, then $(\Lambda^!)^{(s)}\cong A^7\left(\frac{\beta}{\alpha}\right)$ if $\alpha\not=0$ and $(\Lambda^!)^{(s)}\cong A^6$ if $\alpha=0$ (see Theorems \ref{A6A70} and \ref{A7q}). Suppose now that $f$ cannot be presented in the form $f=\sum\limits_{i=0}^{s-1}\alpha^i\beta^{s-i-1}y_1^igy_1^{s-i-1}$ or in the form $f=g^s+\alpha y_1^s$.
Then it follows from Theorems \ref{A4A5}, \ref{A6A70}, \ref{impos_struc} and \ref{A7q} that $(\Lambda^!)^{(s)}\not\cong A^2(q),A^3,A^4,A^5,A^6,A^7(q),A^8,A^9$ for any $q\in\kk$. On the other hand, it is true that $\dim_\kk\Lambda^!_n\not=0$ for any $n\ge 0$, and hence $(\Lambda^!)^{(s)}\not\cong A_0$. It follows from our classification of quadratic algebras with two generators that $(\Lambda^!)^{(s)}\cong A^1$.
\end{proof}

\subsection{$s$-homogeneous triples with $A=A^0$ and $Z_2\ds_AM$}\label{CY}

Let us now consider algebras $\Lambda$ with $\FF(\Lambda^!)$ of the form $(A^0,M,\vp)$ with $M_A\cong Z_2\oplus Z_1^m$. As it was mentioned above, we have ${}_AM\cong Z_2'\oplus (Z_1')^m$ in this case. Here and until the end of this section $A$ denotes the algebra $A^0$. Let us introduce the bimodule
$$
D=(A^{\rm op}\ot A)^2/\langle yf_1,f_1y,xf_2,f_2x, f_1x-xf_1, f_2y-yf_2, f_1x-f_2y\rangle_{A^{\rm op}\ot A},
$$
where we again denote by $f_1$ and $f_2$ the standard generators of $(A^{\rm op}\ot A)^2$. Let us now prove the following lemma recovering the $A$-bimodule structure on $M$.

\begin{lemma}\label{bimod_Z2}
If $(A^0,M,\vp)$ is an $s$-homogeneous triple with $M_A\cong Z_2\oplus Z_1^m$, then $M\cong {}_{\nu}D\oplus \kk^m$ for some graded automorphism $\nu$ of $A^0$. Moreover, the map $\vp:M^{\ot_As}\rightarrow A[1]$ can be factored through the canonical projection $M^{\ot_As}\twoheadrightarrow ({}_{\nu}D)^{\ot_As}$
\end{lemma}
\begin{proof} It is not difficult to see that any two dimensional graded $A$-bimodule $K$ such that $K_A\cong Z_2$ and ${}_AK\cong Z_2'$ is isomorphic to ${}_{\nu}D$ for some graded automorphism $\nu$ of $A$.
Thus, for the first part it is enough to prove that $M$ has an $A$-bimodule summand isomorphic to $\kk^m$. Let us consider the right $A$-module decomposition $M_A=Z_2\oplus Z_1^m$.
We will prove that in fact $Z_1^m$ is an $A$-bimodule summand isomorphic to $\kk^m$. Let us take some $f\in Z_1^m$. Suppose that $uf\not=0$ for some $u\in A_1$. Note that then we have $0\not=\vp(X)f=(1_M\ot_A\vp)(X\ot f)$ for some $X\in M_0^{\ot s}$; in particular, $\vp(Y\ot f)\not=0$ for some $Y\in M_0^{\ot(s-1)}$. On the other hand, $\vp(Y_1\ot f\ot Y_2)=0$ for any $1\le i\le s$, $Y_1\in M_0^{\ot (i-1)}$, $Y_2\in M_0^{\ot (s-i)}$,  and  $f\in Z_1^m$ by Lemma \ref{ker_vp}. Note that the last mentioned assertion implies the second statement of the lemma too.
\end{proof}



\begin{theorem}\label{A0Z}
\begin{enumerate}
\item Suppose that $\Lambda=T_\kk V/(f_1, f_2)$, where $f_1$ and $f_2$ are two linearly independent elements of $V^{\ot s}$.
If $(\Lambda^!)^{(s)}\cong\kk\langle x,y\rangle/(x,y)^2=A^0$ and $(\Lambda^!)^{(s,1)}\cong Z_2\oplus Z_1^m$ as a right $A^0$-module, then $\Lambda\cong \kk\langle x_1,\dots,x_m\rangle*\bD(w)$ for some twisted potential $w$ of degree $s+1$ on a two dimensional space.

\item Suppose that $\Lambda=\kk\langle x_1,\dots,x_m\rangle*\bD(w)$ for a $\nu$-twisted potential $w$ of degree $s+1$ on the two dimensional space $U$. Then $(\Lambda^!)^{(s)}\cong A^0$ and $(\Lambda^!)^{(s,1)}\cong Z_2\oplus Z_1^m$ as a right $A^0$-module if and only if the space $\{(f\ot 1_{U^{\ot s}})w\mid f\in U^*\}$ is two dimensional and there are no nonzero $f_1,f_2\in U^*$ and $u_1,u_2\in U$ such that one of the following conditions holds:
\begin{itemize}
\item $s=2t+1$, $(f_1\ot 1_{U^{\ot s}})w=(u_1\ot u_2)^{\ot t}\ot u_1$, and $(f_2\ot 1_{U^{\ot s}})w=(u_2\ot u_1)^{\ot t}\ot u_2$;
\item $s=2t$, $(f_1\ot 1_{U^{\ot s}})w=(u_1\ot u_2)^{\ot t}$, and $(f_2\ot 1_{U^{\ot s}})w=(u_2\ot u_1)^{\ot t}$.
\end{itemize}
\end{enumerate}
\end{theorem}
\begin{proof} Due to Lemma \ref{bimod_Z2}, we may set $M={}_{\nu^{-1}}D\oplus \kk^m$ for some automorphsm $\nu$ of $A$. As usually, we have $\Lambda=T_\kk M_0^*/(\vp^*A_1^*)$. Let $f_1,f_2$ be the standard basis of ${}_{\nu^{-1}}D$ and $g_1,\dots,g_m$ be a basis of $\kk^m$ considered as direct summands of $M$. Let  $y_1,y_2,x_1,\dots,x_m$ be the basis of $M_0^*$ dual to the obtained basis of $M_0$ and $U$ be the subspace of $M_0^*$ spanned by $y_1$ and $y_2$. It follows from Lemma \ref{bimod_Z2} that $\vp^*A_1^*\subset U^{\ot s}$. Thus, $\Lambda=\kk\langle x_1,\dots,x_m\rangle*T_\kk U/(\vp^*A_1^*)$ and it remains to show that $T_\kk U/(\vp^*A_1^*)=\bD(w)$ for some twisted potential $w$ of degree $s+1$ on $U$.  Let $f_x,f_y$ be the basis of $A_1^*$ dual to $x,y$. Let us define the twisted potential $w$ by the equality $w=y_1\vp^*(f_x)+y_2\vp^*(f_y)$. We can see that
$$\bD(w)=\kk\langle y_1,y_2\rangle/\big((f_1\ot 1_{U^{\ot s}})w,(f_2\ot 1_{U^{\ot s}})w\big)=\kk\langle y_1,y_2\rangle/\big(\vp^*(f_x),\vp^*(f_y)\big)=T_\kk U/(\vp^*A_1),$$ 
and hence, to finish the proof of the first part, it remains to show that $w$ is a $\nu$-twisted potential, where $\nu:U\rightarrow U$ is the composition $U\xrightarrow{\theta}A_1\xrightarrow{\nu}A_1\xrightarrow{\theta^{-1}}U$ with the isomorphism $\theta:U\rightarrow A_1$ defined by the equalities $\theta(y_1)=x$ and $\theta(y_2)=y$. To show that $w$ is a $\nu$-twisted potential it suffices to check the equality
$y_1\vp^*(f_x)+y_2\vp^*(f_y)=\vp^*(f_x)\nu^{-1}(y_1)+\vp^*(f_y)\nu^{-1}(y_2)$. To see this, let us pick some $X\in ({}_{\nu^{-1}}D)^{\ot(s+1)}_0$ and apply both parts of the required equality to it. Direct calculations show that
\begin{multline*}
    \big(y_1\vp^*(f_x)+y_2\vp^*(f_y)\big)(X)f_1x=(1_M\ot_A\vp)X\\
    =(\vp\ot_A1_M)X=\big(\vp^*(f_x)\nu^{-1}(y_1)+\vp^*(f_y)\nu^{-1}(y_2)\big)(X)f_1x,
\end{multline*}
i.e. the applications of $y_1\vp^*(f_x)+y_2\vp^*(f_y)$ and $\vp^*(f_x)\nu^{-1}(y_1)+\vp^*(f_y)\nu^{-1}(y_2)$ give the same result on any element $X\in ({}_{\nu^{-1}}D)^{\ot(s+1)}_0$. Thus, the required equality is proved and $w$ is a $\nu$-twisted potential.

Let us now suppose that $\Lambda=\kk\langle x_1,\dots,x_m\rangle*\bD(w)$ for a $\nu$-twisted potential $w$ of degree $s+1$ on the two dimensional space $U$. Then $\Lambda$ is an $s$-homogeneous algebra with defining ideal generated by $\dim_\kk\{(f\ot 1_{U^{\ot s}})w\mid f\in U^*\}$ elements of degree $s$. Thus, the quadratic algebra $(\Lambda^!)^{(s)}$ has two generators if and only if the last mentioned dimension equals $2$. Note that one of the conditions  listed in the theorem is satisfied for some $f_1,f_2\in U^*$ and linearly dependent $u_1,u_2\in U$ if and only if $\Lambda\cong\kk\langle x_1,\dots,x_m,y_1,y_2\rangle/(y_2^s,f)$ for some homogeneous polynomial $f\in\kk\langle y_1,y_2\rangle_s$. Then it follows from Theorems \ref{A4A5}, \ref{A6A70}, \ref{A7q}, and \ref{A1} that $f_1,f_2\in U^*$ and $u_1,u_2\in U$ satisfying the conditions listed in the theorem exist if and only if $(\Lambda^!)^{(s)}$ is isomorphic to one of the algebras $A^1$, $A^4$, $A^5$, $A^6$, and $A^7(q)$ ($q\in\kk$). Since $(\Lambda^!)^{(s)}\not\cong A^2(q),A^3,A^8,A^9(q)$ for any $q\in\kk$ by Theorem \ref{impos_struc}, we have $(\Lambda^!)^{(s)}\cong A^0$ if and only if the space $\{(f\ot 1_{U^{\ot s}})w\mid f\in U^*\}$ is two dimensional and there are no $f_1,f_2\in U^*$ and $u_1,u_2\in U$ satisfying one of the conditions listed in the theorem.

It remains to show that if $\Lambda=\kk\langle x_1,\dots,x_m\rangle*\bD(w)$ for a $\nu$-twisted potential $w$ of degree $s+1$ on the two dimensional space $U$ and $(\Lambda^!)^{(s)}= A^0$, then $M=(\Lambda^!)^{(s,1)}$ cannot be isomorphic to $B_1(\alpha,\beta)\oplus Z_1^{m+1}$ ($(\alpha,\beta)\in\mathbb{P}^1$) or $Z_1^{m+2}$ as a right $A^0$-module. Let $y_1,y_2$ be a basis of $U$ and $f_1,f_2$ be the basis of $U^*$ dual to it. Note that by definition $\bD(w)=\kk\langle y_1,y_2\rangle/(R)$, where $R$ is the subspace of $U^{\ot s}$ generated by $r_1=(f_1\ot 1_{U^{\ot s}})w$ and $r_2=(f_2\ot 1_{U^{\ot s}})w$. Note that $w=y_1r_1+y_2r_2=r_1\nu^{-1}(y_1)+r_2\nu^{-1}(y_2)\in (U\ot R)\cap (R\ot U)$. On the other hand, $M_A\cong Z_1^{m+2}$ means $\Lambda^!_{s+1}=0$, and hence $U^*\ot R^{\perp}+R^{\perp}\ot U^*=(U^*)^{\ot(s+1)}$ and $(U\ot R)\cap (R\ot U)=0$. Here, as usually, $R^{\perp}=\{f\in(U^*)^{\ot s}\mid f(R)=0\}$. At the same time,  $M_A\cong B_1(\alpha,\beta)\oplus Z_1^{m+1}$ means that there is $f\in (U^*)^{\ot s}$ such that $f(r_1)$ or $f(r_2)$ is nonzero, but $g\ot f\in U^*\ot R^{\perp}+R^{\perp}\ot U^*$ for any $g\in U^*$. Since $h\big((U\ot R)\cap (R\ot U)\big)=0$ for any $h\in U^*\ot R^{\perp}+R^{\perp}\ot U^*$, we have $0=(g\ot f)(w)=g(y_1)f(r_1)+g(y_2)f(r_2)$ for any $g\in U^*$. This obviously contradicts to the fact that $f(r_1)$ and $f(r_2)$ cannot be zero simultaneously, and thus $M_A\not\cong B_1(\alpha,\beta)\oplus Z_1^{m+1}$. The theorem is proved.
\end{proof}

Let us now describe the automorphisms $\nu$ of the algebra $A^0$ for which there is an $s$-homogeneous triple with $(A,M)=(A^0,{}_{\nu^{-1}} D)$. As it was proved above, this is equivalent to the existence of $\nu$-twisted potential $w$ of degree $s+1$ on a two dimensional space with the basis $y_1,y_2$ that does not satisfy conditions listed in item 2 of Theorem \ref{A0Z}.
Modulo an automorphism of $A^0$ that induces an isomorphism of the corresponding $s$-homogeneous triples, we may assume that either $\nu(y_1)=\lambda_1y_1$ and $\nu(y_2)=\lambda_2 y_2$ for some $\lambda_1,\lambda_2\in\kk^*$ or $\nu(y_1)=\lambda y_1$ and $\nu(y_2)=\lambda (y_1+y_2)$ for some $\lambda\in\kk^*$. 
As usually, we consider only the case $s>2$ while there are no such potentials in the case $s=2$, for example, due to the classification \cite[Theorem 3.10.6]{Ufn}.

\begin{lemma}\label{CYs}
Suppose that $\nu(y_1)=\lambda_1y_1$ and $\nu(y_2)=\lambda_2 y_2$ for some $\lambda_1,\lambda_2\in\kk^*$. An $s$-homogeneous algebra $\Lambda$ such that $\FF(\Lambda^!)=(A^0,{}_{\nu^{-1}} D,\vp)$ for some $\vp$ exists if and only if one of the following conditions holds:
\begin{itemize}
    \item there is some $2\le k\le s-1$ such that $\lambda_1^k\lambda_2^{s+1-k}=1$;
    \item $\lambda_1\lambda_2^s=\lambda_1^s\lambda_2=1$;
    \item $\lambda_1=1$, $\lambda_2^{s}=1$ (or $\lambda_1^{s}=1$, $\lambda_2=1$).
\end{itemize}
\end{lemma}
\begin{proof} Suppose that $w$ is a $\nu$-twisted potential of degree $s+1$. Note first that $w=\phi_{\nu}^{s+1}(w)=\nu^{\ot (s+1)}(w)$. Thus, $\nu^{\ot (s+1)}$ has to have an eigenvalue equal to $1$. Thus, $\lambda_1^k\lambda_2^{s+1-k}=1$ for some $0\le k\le s+1$. Note that $\lambda_1^{s+1}$ and $\lambda_2^{s+1}$ correspond  to the eigenvectors $y_1^{s+1}$ and $y_2^{s+1}$ respectively. Note also that the vector $y_1^{s+1}$ can appear only in the case where $\lambda_1=1$ while $y_2^{s+1}$ can appear only in the case where $\lambda_2=1$. Thus, if none of the conditions of the lemma holds, then either $w=ay_1^{s+1}+b\sum\limits_{i=0}^{s+1}\lambda_1^iy_1^iy_2y_1^{s-i}$ or $w=ay_2^{s+1}+b\sum\limits_{i=0}^{s+1}\lambda_2^iy_2^iy_1y_2^{s-i}$. In both cases one can show that $w$ satisfies one of the conditions listed in item 2 of Theorem \ref{A0Z}, and hence does not determine an $s$-homogeneous triple of the required form.

If we have $\lambda_1^k\lambda_2^{s+1-k}=1$ for some $2\le k\le s-1$, then the required $\nu$-twisted potential is $w=\sum\limits_{i=0}^{k-1}\lambda_1^iy_1^iy_2^{s+1-k}y_1^{k-i}+\lambda_1^k\sum\limits_{i=0}^{s+1-k}\lambda_2^iy_2^iy_1^ky_2^{s+1-k-i}$.
If $\lambda_1\lambda_2^s=\lambda_1^s\lambda_2=1$ and $\lambda_1^k\lambda_2^{s+1-k}\not=1$ for any $2\le k\le s-1$, then the required $\nu$-twisted potential is $w=\sum\limits_{i=0}^{s+1}\lambda_1^iy_1^iy_2y_1^{s-i}+\sum\limits_{i=0}^{s+1}\lambda_2^iy_2^iy_1y_2^{s-i}$.
If $\lambda_1=1$ and $\lambda_2^{s}=1$, then the required $\nu$-twisted potential is $w=y_1^{s+1}+\sum\limits_{i=0}^{s}\lambda_2^iy_2^iy_1y_2^{s-i}$. It is easy to see that the listed potentials do not satisfy the conditions listed in item 2 of Theorem \ref{A0Z}.
\end{proof}

\begin{lemma}\label{CYns}
Suppose that $\nu(y_1)=\lambda y_1$ and $\nu(y_2)=\lambda(y_1+y_2)$ for some $\lambda\in\kk^*$. If $s>3$, then an $s$-homogeneous algebra $\Lambda$ such that $\FF(\Lambda^!)=(A^0,{}_{\nu^{-1}} D,\vp)$ for some $\vp$ exists if and only if $\lambda^{s+1}=1$. A $3$-homogeneous algebra $\Lambda$ satisfying the same condition exists if and only if $\lambda^2=1$.
\end{lemma}
\begin{proof} Suppose that $w$ is a $\nu$-twisted potential of degree $s+1$. Note that all the eigenvalues of $\nu^{\ot (s+1)}$ are equal to $\lambda^{s+1}$, and hence $\lambda^{s+1}=1$ (see the proof of Lemma \ref{CYs}).

Let us consider the case $s=3$. If $\lambda=1,-1$, then the required $\nu$-twisted potential is
$$w=y_1^2y_2^2+y_2^2y_1^2+\lambda(y_1y_2^2y_1+y_2y_1^2y_2)+(\lambda-1)\big((y_1y_2)^2+(y_2y_1)^2\big)+y_1y_2y_1^2-y_1^2y_2y_1.$$
Suppose now that $\lambda^2=-1\not=1$. Direct calculations show that any eigenvector of $\nu^{\ot 4}$ is a linear combination of $y_1^4$, $\sum\limits_{i=0}^3a_iy_1^iy_2y_1^{3-i}$ with $\sum\limits_{i=0}^3a_i=0$, $y_1^2y_2^2+y_2^2y_1^2-y_1y_2^2y_1-y_2y_1^2y_2$ and $y_1y_2^2y_1+y_2y_1^2y_2-(y_1y_2)^2-(y_2y_1)^2$.
In particular, $(y_1y_2)^2$ and $(y_2y_1)^2$ have equal coefficients in $w$ and the coefficient of $(y_2y_1)^2$ in $\phi_{\nu}(w)$ equals to the coefficient of $(y_1y_2)^2$ in $w$ multiplied by $\lambda$. Thus, the coefficient of $y_1y_2^2y_1+y_2y_1^2y_2-(y_1y_2)^2-(y_2y_1)^2$ in the linear combination mentioned above is zero. Then the coefficient of $y_2y_1^2y_2$ in $w$ equals to the coefficient of $y_1^2y_2^2$ in $w$ multiplied by $-1$ and, at the same time, the coefficient of $y_2y_1^2y_2$ in $\phi_{\nu}(w)$ equals to  the coefficient of $y_1^2y_2^2$ in $w$ multiplied by $\lambda$. Thus, $w$ is a linear combination of $y_1^4$, $\sum\limits_{i=0}^3a_iy_1^iy_2y_1^{3-i}$ with $\sum\limits_{i=0}^3a_i=0$. Then it is easy to see that $w=a\big(y_1^4+(1-\lambda)\sum\limits_{i=0}^3\lambda^iy_1^iy_2y_1^{3-i}\big)$ for some $a\in\kk$, and hence satisfies one of the conditions listed in item 2 of Theorem \ref{A0Z}.

Suppose now that $s>3$ and $\lambda^{s+1}=1$. To define a $\nu$-twisted potential, we will consider three cases
\begin{enumerate}
\item If $\lambda^2+1\not=0$, then
\begin{multline*}
w=(1+\lambda^2)\sum\limits_{i=0}^{s-1}\lambda^iy_1^iy_2^2y_1^{s-1-i}+(\lambda^s+\lambda)y_2y_1^{s-1}y_2-(1+\lambda)\sum\limits_{i=0}^{s-2}\lambda^{i+1}y_1^iy_2y_1y_2y_1^{s-2-i}\\
-(\lambda^s+1)y_2y_1^{s-2}y_2y_1-(1+\lambda)y_1y_2y_1^{s-2}y_2-(\lambda^s+1)y_1^{s-1}y_2y_1\\
+(\lambda^s-1)y_1^{s}y_2+(1-\lambda)y_2y_1^{s}+(1+\lambda)y_1y_2y_1^{s-1}+y_1^{s+1}.
\end{multline*}
\item If $\charr\kk=2$ and $\lambda=1$, then
$$
w=\sum\limits_{i=0}^{s-1}y_1^iy_2^2y_1^{s-1-i}+y_2y_1^{s-1}y_2+y_1^{s}y_2+y_2y_1^{s}.
$$
\item If $\lambda^2+1=0$ and $\charr\kk\not=2$, then we have $\lambda^3+1\not=0$. Note also that in this case $4\mid s+1$, i.e. $s\ge 7$. In this case we define
\begin{multline*}
w=(1+\lambda^3)\sum\limits_{i=0}^{s-1}\lambda^iy_1^iy_2^2y_1^{s-1-i}+(\lambda^s+\lambda^2)y_2y_1^{s-1}y_2-(1+\lambda)\sum\limits_{i=0}^{s-3}\lambda^{i+2}y_1^iy_2y_1^2y_2y_1^{s-3-i}\\
-(\lambda^s+1)y_2y_1^{s-3}y_2y_1^2-(1+\lambda)y_1y_2y_1^{s-3}y_2y_1-(\lambda+\lambda^2)y_1^2y_2y_1^{s-3}y_2\\
-(\lambda^s+\lambda^2)y_1^{s-2}y_2y_1^2-(1+\lambda^3)y_1^{s-1}y_2y_1+(\lambda^s-\lambda)y_1^{s}y_2+(1-\lambda^2)y_2y_1^{s}\\
+(1+\lambda)y_1y_2y_1^{s-1}+(\lambda+\lambda^2)y_1y_2y_1^{s-1}+(1+\lambda)y_1^{s+1}.
\end{multline*}
\end{enumerate}
\end{proof}

\subsection{$s$-homogeneous triples with $A=A^0$ and $B_1(\alpha,\beta)\ds_AM$}\label{Bsec}

Let us now consider algebras $\Lambda$ with $\FF(\Lambda^!)$ of the form $(A^0,M,\vp)$ with $M_A\cong B_1(\alpha,\beta)\oplus Z_1^m$ ($(\alpha,\beta)\in\mathbb{P}_1$).
As it was mentioned above, we may assume in this case that $(\alpha,\beta)=(1,0)$ and ${}_AM\cong B_1'(0,1)\oplus (Z_1')^m$. Note that $m>0$ follows from the definition of an $s$-homogeneous triple.
Let us introduce the $A$-bimodules $B_1=(A^{\rm op}\ot A)/\langle f_1x,yf_1,f_1y-xf_1\rangle_{A^{\rm op}\ot A}$ and $B_2=(A^{\rm op}\ot A)^2/\langle f_1x,xf_1,yf_1,yf_2,f_2x,f_2y,f_1y-xf_2\rangle_{A^{\rm op}\ot A}$, where $f_1,f_2$ as usually denote the standard generators. It is not difficult to see that there are only two possible $A$-bimodule structures on $M$; namely, we may have $M= B_1\oplus\kk^m$ or $M= B_2\oplus\kk^{m-1}$. This reduces the first step that we usually fulfill in our proofs, and thus we can go directly to the description of algebras $\Lambda$ corresponding to the $s$-homogeneous triple $(A^0,M,\vp)$.

\begin{theorem}\label{A0B} Suppose that $\Lambda=T_\kk V/(f_1, f_2)$, where $f_1$ and $f_2$ are two linearly independent elements of $V^{\ot s}$.
\begin{enumerate}
    \item If $(\Lambda^!)^{(s)}\cong\kk\langle x,y\rangle/(x,y)^2=A^0$ and $(\Lambda^!)^{(s,1)}\cong B_1\oplus\kk^m$ as an $A^0$-bimodule, then $\Lambda\cong \kk\langle x_1,\dots,x_m,y_1\rangle/(y_1f,fy_1)$ for some $f\in \kk\langle x_1,\dots,x_m,y_1\rangle_{s-1}$. Conversely, if $\Lambda=\kk\langle x_1,\dots,x_m,y_1\rangle/(y_1f,fy_1)$, then $(\Lambda^!)^{(s)}\cong A^0$ and $(\Lambda^!)^{(s,1)}\cong B_1\oplus\kk^m$ as an $A^0$-bimodule except the case where $s=2t$ and $f=(gy_1)^{t-1}g$ for some $g\in \kk\langle x_1,\dots,x_m,y_1\rangle_1$.

\item If $(\Lambda^!)^{(s)}\cong\kk\langle x,y\rangle/(x,y)^2=A^0$ and $(\Lambda^!)^{(s,1)}\cong B_2\oplus\kk^m$ as an $A^0$-bimodule, then $\Lambda\cong \kk\langle x_1,\dots,x_m,y_1,y_2\rangle/(y_1f,fy_2)$ for some $f\in \kk\langle x_1,\dots,x_m,y_1,y_2\rangle_{s-1}$. Conversely, if $\Lambda=\kk\langle x_1,\dots,x_m,y_1,y_2\rangle/(y_1f,fy_2)$, then $(\Lambda^!)^{(s)}\cong A^0$ and $(\Lambda^!)^{(s,1)}\cong B_2\oplus\kk^m$ as an $A^0$-bimodule except the case where $s=2t+1$ and $f=\alpha (y_2y_1)^t$ for some $\alpha\in\kk^*$ and the case where $f=\alpha y_1^{s-1}$ or $f=\alpha y_2^{s-1}$ for some $\alpha\in\kk^*$.
\end{enumerate}
\end{theorem}
\begin{proof} Suppose first that $F(\Lambda^!)=(A^0,B_k\oplus\kk^m,\vp)$ for some homomorphism $\vp$ and some $k\in\{1,2\}$. Let now $x_1,\dots,x_m,y_1,\dots,y_k\in M_0^*$ be the basis dual to a basis $g_1,\dots,g_m,f_1,\dots,f_k$ of $M_0$ whose first $m$ elements constitute a basis of the summand $\kk^{m}$ and the last $k$ are standard generators of the summand $B_k$. We will denote $\kk\langle x_1,\dots,x_m,y_1,\dots,y_k\rangle$ by $C$ for short during this proof.
As usually, we have $\Lambda\cong C/(\vp^*A_1^*)$.
Let now $f_x,f_y$ be the basis of $A_1^*$ dual to the basis $x,y$ of $A_1$. Note that if $f_x\vp(r\ot X)\not=0$ for some $r\in M_0$ such that $rx=ry=0$ and some $X\in M_0^{\ot (s-1)}$, then $0\not=\vp(r\ot X)f_k=r\vp(X\ot f_k)=0$. Hence, $\vp^*f_x=y_1f$ for some $f\in C_{s-1}$. Analogously, $\vp^*f_y=f'y_k$ for some $f'\in C_{s-1}$. For any $X\in M_0^{\ot (s-1)}$, we have
$f(X)xf_k=\vp(f_1\ot X)f_k=f_1\vp(X\ot f_k)=f'(X)f_1y$, and hence $f'=f$.

Suppose now that $\Lambda=C/(y_1f,fy_k)$, where $C$ is as above and $y_1f$ and $fy_k$ are linearly independent $s$-homogeneous polynomials.
Let us prove first that $(\Lambda^!)^{(s)}\cong A^0$ except the cases listed in the theorem. Suppose that it is not true. Then due to Lemma \ref{class} and Theorems \ref{A4A5}, \ref{A6A70}, \ref{impos_struc}, \ref{A7q}, and \ref{A1} there are some nonzero (but maybe linearly dependent) $g_1,g_2\in C_1$ such that $(g_1g_2)^t,(g_2g_1)^t\in (y_1f,fy_k)$ in the case $s=2t$ and $(g_1g_2)^tg_1,(g_2g_1)^tg_2\in (y_1f,fy_k)$ in the case $s=2t+1$. Let us recall that the tensor $w\in V^{\ot s}$ is called {\it decomposable} if $w=w_1\ot\dots\ot w_s$ for some $w_i\in V$ ($1\le i\le s$).
Suppose that the tensor $\beta y_1f+\gamma fy_k$ is decomposable for some $\beta,\gamma\in\kk^*$. Let us present $f$ in the form $f=y_1^pf'y_k^q$, where $f'\in C_{s-1-p-q}$ can be presented neither in the form $f'=y_1f''$ nor in the form $f'=f''y_k$ with $f''\in C$. The tensor $\beta y_1f'+\gamma f'y_k$ has to be decomposable.
If $p+q\not=s-1$, then it is easy to see that there is some $h_1\in C_{s-1-p-q}^*$ such that $h_1(f')=1$ and $h_1(f''y_k)=0$ for any $f''\in C_{s-2-p-q}$. Then we have $(1_{C_1}\ot h_1)(\beta y_1f'+\gamma f'y_k)=\beta y_1$. On the other hand, there is some $h_2\in C_{s-1-p-q}^*$ such that $(1_{C_1}\ot h_2)(\beta y_1f'+\gamma f'y_k)$ is linearly independent with $y_1$, and hence $\beta y_1f'+\gamma f'y_k$ cannot be decomposable. Thus, $\beta y_1f+\gamma fy_k$ can be decomposable only for $f$ of the form $f=\alpha y_1^py_k^{s-1-p}$ with $\alpha\in\kk^*$. It is easy to see that if $f$ has such a form and is not equal to $\alpha y_1^{s-1}$ or $\alpha y_k^{s-1}$, then there are no nonzero linear $g_1,g_2$ that satisfy the condition mentioned above. On the other hand, the cases $f=\alpha y_1^{s-1}$ and $f=\alpha y_k^{s-1}$ are excluded by our assumptions. Our argument shows that if $(\Lambda^!)^{(s)}\not\cong A^0$, then $f$ is decomposable. Since $f$ is not linearly dependent with a power of $y_1$ or $y_k$, one of the elements $g_1,g_2$ mentioned above is linearly dependent with $y_1$ and, in the case $k=2$, the second one is linearly dependent with $y_2$. Then it is easy to see that one of the conditions prohibited by the theorem holds. Thus, we have proved that $(\Lambda^!)^{(s)}\cong A^0$.

Let us denote by $M$ the $A$-bimodule $(\Lambda^!)^{(s,1)}$.
Let us denote by $R\subset C_s$ the subspace generated by $y_1f$ and $fy_k$. Note that $y_1(fy_k)=(y_1f)y_k\in RC_1\cap C_1R$, and hence $M\not\cong \kk^{m+k}$ (see the proof of Theorem \ref{A0Z}). Let us now prove that $M$ cannot contain right $A$-module summand isomorphic to $Z_2$. Due to Theorem \ref{A0Z} it is enough to show that there is no twisted potential in two variables $w\in C_{s+1}$ such that $\{(h\ot 1_{C_s})w\mid h\in C_1^*\}=\langle y_1f,fy_k\rangle$.
Here we consider the cases $k=1$ and $k=2$ separately. If $k=2$, then $w,f\in \kk\langle y_1,y_2\rangle$. Using the conditions above, we get $w=y_1(\alpha_1y_1f+\beta_1fy_2)+y_2(\alpha_2y_1f+\beta_2fy_2)$ for some $\alpha_1,\alpha_2,\beta_1,\beta_2\in\kk$. Since $w$ is a $\nu$-twisted potential for some automorphism $\nu$ of $\kk\langle y_1,y_2\rangle_1$, we have 
\begin{multline*}
    y_1(\alpha_1y_1f+\beta_1fy_2)+y_2(\alpha_2y_1f+\beta_2fy_2)=(\alpha_1y_1f+\beta_1fy_2)\nu^{-1}(y_1)+(\alpha_2y_1f+\beta_2fy_2)\nu^{-1}(y_2)\\
    =(\alpha_1'y_1f+\beta_1'fy_2)y_1+(\alpha_2'y_1f+\beta_2'fy_2)y_2
\end{multline*}
for some $\alpha_1',\alpha_2',\beta_1',\beta_2'\in\kk$. Let us represent $f$ in the form 
$$f=\sum\limits_{\scriptsize\begin{array}{c}
     i_1+\dots+i_t+j_1+\dots+j_t=s-1\\
     i_2,\dots,i_t,j_1,\dots,j_{t-1}>0
\end{array}}c_{i_1,\dots,i_t,j_1,\dots,j_t}y_1^{i_1}y_2^{j_1}\dots y_1^{i_t}y_2^{j_t}.$$
Choosing maximal $j$ such that there is $c_{i_1,\dots,i_t,j_1,\dots,j_t}\not=0$ with $j_t=j$ and looking at the equality above, we get $\beta_2'=0$. In a similar way, one can show that $\alpha_1=\beta_2=\alpha_1'=0$, and hence
$\alpha_2y_2y_1f=\beta_1'fy_2y_1+(\alpha_2'-\beta_1)y_1fy_2$. Since all monomials in the decomposition of $y_1fy_2$ start with $y_1$ and end with $y_2$, it follows that $\alpha_2'=\beta_1$. If $\alpha_2=0$, then $w=\beta_1y_1fy_2$, and hence $\{(h\ot 1_{C_s})w\mid h\in C_1^*\}$ is one dimensional. Thus, we have $y_2y_1f=\gamma fy_2y_1$ for some $\gamma\in\kk$. Hence, $s=2t+1$ and $f$ equals $(y_2y_1)^t$ in this case.

Let us now consider the case $k=1$. In this case $w,f\in \kk\langle y_1,z\rangle$ for some $z\in C_1$ and, analogously to the case $k=2$, we get
$$
    w=y_1(\alpha_1y_1f+\beta_1fy_1)+z(\alpha_2y_1f+\beta_2fy_1)=(\alpha_1'y_1f+\beta_1'fy_1)y_1+(\alpha_2'y_1f+\beta_2'fy_1)z.
$$
for some $\alpha_1,\alpha_2,\beta_1,\beta_2m\alpha_1',\alpha_2',\beta_1',\beta_2'\in\kk$. Analogously to the case $k=2$, one can show that $\alpha_1=\beta_2=\beta_1'=\alpha_2'=0$, and hence
$\alpha_2zy_1f=\beta_2'fy_1z+(\alpha_1'-\beta_1)y_1fy_1$. Now, as before, we get $\alpha_1'=\beta_1$ and $zy_1f=\gamma fy_1z$ for some $\gamma\in\kk$. Hence, $s=2t$ and $f$ equals $(zy_1)^{t-1}z$  in this case. Thus, we have proved that $\Lambda$ is not an algebra defined by a twisted potential on a two dimensional space.

It remains to prove that $M$ cannot be isomorphic to $B_{3-k}\oplus \kk^{m+2k-3}$. Due to the facts that we have already proved, it suffices to show that there are no nonzero
$f,f'\in \kk\langle x_1,\dots,x_m,y_1,y_2\rangle_{s-1}$ and linear polynomial $g$ such that $\langle gf,fg\rangle=\langle y_1f',f'y_2\rangle$. Suppose that we have found $f$, $f'$ and $g$ that satisfy the just mentioned conditions.
We may assume without loss of generality that $g$ and $y_1$ are linearly independent. Indeed, in the opposite case $g$ is linearly independent with $y_2$ and one can simply apply the functor ${}^{\rm op}$ and interchange $y_1$ and $y_2$.
Then obviously $f'\not=y_1f''$ for any $f''\in \kk\langle x_1,\dots,x_m,y_1,y_2\rangle_{s-2}$. It is easy to check that $f'\not=\alpha y_2^{s-1}$ for any $\alpha\in\kk$, and hence $\beta y_1f'+\gamma f'y_2$ cannot be presented in the form $g'f''$ for some linear $g'$ if $\beta,\gamma\in\kk^*$ (see the proof that this tensor cannot be decomposable above).
Thus, we may assume that $f'=gf''$ and $f=f''y_2$ for some $f''\in \kk\langle x_1,\dots,x_m,y_1,y_2\rangle_{s-2}$. Then we have $\langle gf''y_2,f''y_2g\rangle=\langle y_1gf'',gf''y_2\rangle$. It is clear that $f''\not=gf'''$  for any $f'''\in \kk\langle x_1,\dots,x_m,y_1,y_2\rangle_{s-3}$. Then $\beta gf''y_2+\gamma f''y_2g$ cannot be presented in the form $g'f'''$ for some linear $g'$ if $\beta,\gamma\in\kk^*$. Thus, we have $f''y_2g=\alpha y_1gf''$ for some $\alpha\in\kk^*$. The obtained equality is not valid for any $f''$, and hence we get a contradiction that finishes the proof.
\end{proof}

\subsection{$s$-homogeneous triples with $A=A^0$ and $M=\kk^m$}

Now we are ready to give a criterion for $s$-homogeneous algebra $\Lambda$ having the $s$-homogeneous triple $\FF(\Lambda^!)$ of the form $(A^0,\kk^m,\vp)$. Note that if $\Lambda$ is an $s$-homogeneous algebra with two relations, then the last condition is equivalent to the equality $\Lambda^!_{s+1}=0$. To achieve our purpose we simply exclude all other situations. Thus, we get the following theorem.

\begin{theorem}\label{A0s}
Suppose that $\Lambda=T_\kk V/(f_1, f_2)$, where $f_1$ and $f_2$ are two linearly independent elements of $V^{\ot s}$. Then $(\Lambda^!)^{(s)}\cong\kk\langle x,y\rangle/(x,y)^2=A^0$ and $(\Lambda^!)^{(s,1)}\cong \kk^m$ as an $A^0$-bimodule if and only if non of the following conditions holds:
\begin{itemize}
    \item There are some nonzero $u,v\in V$ such that either $s=2t$ and $(u\ot v)^{\ot t},(v\ot u)^{\ot t}\in\langle f_1,f_2\rangle$ or $s=2t+1$ and $(u\ot v)^{\ot t}\ot u,(v\ot u)^{\ot t}\ot v\in \langle f_1,f_2\rangle$.
    \item There is a two dimensional subspace $U$ of $V$ and some twisted potential $w\in U^{\ot (s+1)}$ such that $\langle f_1,f_2\rangle=\{(f\ot 1_{U^{\ot s}})w\mid f\in U^*\}$.
    \item There are some $f\in V^{\ot (s-1)}$ and $u,v\in V$ such that $\langle f_1,f_2\rangle=\langle u\ot f,f\ot v\rangle$.
\end{itemize}
\end{theorem}

\subsection{About the $s$-Koszulity}

In this subsection we discuss the $s$-Koszulity of algebras considered in this paper, i.e. connected $s$-homogeneous algebras with two relations. It follows from what is already proved in this paper that $\Lambda$ can be not $s$-Koszul only in the cases where $A=A^1$ or $A=A^0$.
We will show that the algebras $\Lambda$ with $A=A^0$ and $M=B_1\oplus\kk^m$ are not $s$-Koszul and moreover do not satisfy the extra condition.
For all the remaining algebras we will prove criteria that simplify the verification of $s$-Koszulity.
Namely, we will show that the $s$-Koszulity follows from the equality \eqref{HSm} for the Hilbert series and is equivalent to the exactness of the generalized Koszul complex in the second term. Note that it was shown in \cite{Pos,Roo} that the equality for Hilbert series does not imply the Koszul property even for quadratic algebras. Moreover, it was shown in \cite{DimP} that there exist two algebras one of which is Koszul and another is not Koszul that have the same Hilbert series.

Let us start with a not difficult consideration that the algebras $\Lambda$ with $(\Lambda^!)^{(s)}\cong\kk\langle x,y\rangle/(x,y)^2=A^0$ and $(\Lambda^!)^{(s,1)}\cong B_1\oplus\kk^m$ are not $s$-Koszul.

\begin{prop}\label{A0B1_K}
Suppose that $\Lambda\cong \kk\langle x_1,\dots,x_m,y_1\rangle/(y_1f,fy_1)$ for some $f\in \kk\langle x_1,\dots,x_m,y_1\rangle_{s-1}$ such that $y_1f$ and $fy_1$ are linearly independent.
Then $\Lambda$ does not satisfy the extra condition and, in particular, is not $s$-Koszul except the case where $s=2t$ and $f=(gy_1)^{t-1}g$ for some $g\in \kk\langle x_1,\dots,x_m,y_1\rangle_1$.
\end{prop}
\begin{proof} We have $M_0^*=\bigoplus\limits_{i=1}^m\kk x_i\oplus \kk y_1$ and $\vp^*A_1^*=\kk fy_1\oplus \kk y_1f$. Our assumptions and Theorem \ref{A0B} guarantee also that $(\vp^*A_1^*)\ot M_0^*\cap M_0^*\otimes (\vp^*A_1^*)=\kk y_1fy_1$. Note that $fy_1f\in (\vp^*A_1^*)\ot (M_0^*)^{\ot(s-1)}\cap (M_0^*)^{\ot(s-1)}\ot (\vp^*A_1^*)$, and hence it is enough to show that $fy_1f\not\in y_1fy_1\kk\langle x_1,\dots,x_m,y_1\rangle_{s-2}$. But the last assertion is clear because $f$ and $y_1^{s-1}$ are linearly independent.
\end{proof}

Let us now prove the main result of this subsection.

\begin{theorem}\label{sKosz}
Suppose that $\Lambda=T_\kk V/(f_1, f_2)$, where $f_1$ and $f_2$ are two linearly independent elements of $V^{\ot s}$. Then $\Lambda$ is $s$-Koszul if and only if its generalized Koszul complex is exact in the second term and if and only if $\Big(\HH_A(t^s)-t\HH_M(t^s)\Big)\HH_{\Lambda}(t)=1$, where $A=(\Lambda^!)^{(s)}$ and $M=(\Lambda^!)^{(s,1)}$.
\end{theorem}
\begin{proof} Note that the grading on $\Lambda$ induces a grading on the terms of its generalized Koszul complex $\cdots\xrightarrow{d_2} K_2\xrightarrow {d_1}K_1\xrightarrow{d_0} K_0$.
We will prove the following assertion. If there is some homogeneous $u\in K_i$, $i>2$ such that $d_{i-1}(u)=0$ and $u\not\in \Im d_i$, then there is sum $u_0\in K_2$ such that $d_1(u_0)=0$, $u_0\not\in\Im d_2$ and $\deg(u_0)<\deg(u)$.
Let us first suppose that this fact is valid. Then the exactness in the second term obviously implies the exactness of $K$ in all terms $i\ge 2$ while the generalized Koszul complex is always exact in the first term and has homology isomorphic to $\kk$ in the zero term.
Suppose now that $K$ is not exact. Let us denote by $K_{i,j}$ the $j$-degree component of $K_i$ with respect to the grading induced by the grading of $\Lambda$. Let us consider minimal $k>0$ such that  the complex $\cdots\xrightarrow{d_2|_{K_{3,k}}} K_{2,k}\xrightarrow {d_1|_{K_{2,k}}}K_{1,k}\xrightarrow{d_0|_{K_{1,k}}} K_{0,k}$ is not exact. Then this complex is exact in all terms except the second one, and hence the coefficient of $t^k$ in $\Big(\HH_A(t^s)-t\HH_M(t^s)\Big)\HH_{\Lambda}(t)$ equals to $\dim_\kk\Ker\left(d_1|_{K_{2,k}}\right)-\dim_\kk\Im\left(d_2|_{K_{3,k}}\right)\not=0$.

Note that the assertion of the theorem is obvious if either $A\not\cong A^0,A^1$ or $A=A^0$, $M=\kk^m$ since in these cases $K$ is automatically exact in degrees $i>2$. If $A=A^0$ and $M=B_1\oplus \kk^m$, than $\Lambda$ does not satisfy the extra condition by Proposition \ref{A0B1_K}, and hence it is not $s$-Koszul, $K$ is not exact in the second term, and $\Big(\HH_A(t^s)-t\HH_M(t^s)\Big)\HH_{\Lambda}(t)\not=1$.

Let us now consider the case where $A=A^1$ and $M=A/(x)\oplus\kk^m$, i.e. $\Lambda=\kk\langle x_1,\dots,x_m,y_1\rangle/(f,y_1^s)$ and $f$ does not satisfy conditions listed in Theorem \ref{A1}. In this case the generalized Koszul complex has the form
\begin{multline*}
\cdots\xrightarrow{\scriptsize y_1}\Lambda(-3s)\xrightarrow{\scriptsize y_1^{s-1}}\Lambda(-2s-1)\xrightarrow{\scriptsize y_1}\Lambda(-2s)\xrightarrow{\scriptsize y_1^{s-1}}\Lambda(-s-1)\xrightarrow{\scriptsize\begin{pmatrix}y_1\\0\end{pmatrix}}\Lambda^2(-s)\\
\xrightarrow{\scriptsize\begin{pmatrix}y_1^{s-1}&0&\cdots&0\\f^{y_1}&f^{x_1}&\cdots&f^{x_m}\end{pmatrix}}\Lambda^{m+1}(-1)\xrightarrow{\scriptsize\begin{pmatrix}y_1&x_1&\cdots&x_m\end{pmatrix}}\Lambda,
\end{multline*}
where $f=y_1f^{y_1}+\sum\limits_{i=1}^mx_if^{x_i}$.
If there is some element $g\in K_{2i,j}=\Lambda(-is)_j=\Lambda_{j-is}$ with $i>1$ such that $d_{2i-1}(g)=0$ and $g\not\in\Im(d_{2i})$, then we have $y_1^{s-1}g=0$ and $g\not\in y_1\Lambda$,
and hence $(g,0)\in K_{2,j-(i-1)s}=\Lambda^2(-s)_{j-(i-1)s}$ satisfies the conditions $d_1(g,0)=0$ and $(g,0)\not\in\Im(d_{2})$. Suppose now that $g\in K_{2i+1,j}=\Lambda(-is-1)_j=\Lambda_{j-is-1}$ with $i\ge 1$ is such that $d_{2i}(g)=0$ and $g\not\in\Im(d_{2i+1})$, i.e. $y_1g=0$ and $g\not\in y_1^{s-1}\Lambda$. Let us choose maximal $0\le l<s-1$ such that $g\in y_1^l\Lambda$, i.e. $g=y_1^lg'$ and $g'\not\in y_1\Lambda$. Then we have $y_1^{s-1}g'=y_1^{s-1-l}g=0$, and hence $(g',0)\in K_{2,j-(i-1)s-k-1}=\Lambda^2(-s)_{j-(i-1)s-k-1}$ satisfies the conditions $d_1(g',0)=0$ and $(g',0)\not\in\Im(d_{2})$.

Suppose now that $A=A^0$ and $M={}_{\nu^{-1}}D\oplus\kk^m$ for some graded automorphism $\nu$ of $A^0$, i.e. $\Lambda=\kk\langle x_1,\dots,x_m\rangle*\bD(w)$ for some $\nu$-twisted potential $w$ on a two dimensional space $U$ with the basis $y_1,y_2$ that does not satisfy conditions listed in Theorem \ref{A0Z}.
Let us introduce $f_{ij}=(f_{y_i}\otimes f_{y_j}\otimes 1^{\otimes (s-1)})w$ for $1\le i,j\le 2$. Then the definition of a $\nu$-twisted potential implies $f_1=y_1f_{11}+y_2f_{12}=f_{11}\nu^{-1}(y_1)+f_{21}\nu^{-1}(y_2)$ and $f_2=y_1f_{21}+y_2f_{22}=f_{12}\nu^{-1}(y_1)+f_{22}\nu^{-1}(y_2)$.
In particular, we have $\bD(w)=\kk\langle y_1,y_2\rangle/(f_1,f_2)$.
In this case the generalized Koszul complex has the form
$$
\Lambda(-s-1)\xrightarrow{\scriptsize\begin{pmatrix}\nu^{-1}(y_1)\\\nu^{-1}(y_2)\end{pmatrix}}\Lambda^2(-s)
\xrightarrow{\scriptsize\begin{pmatrix}f_{11}&f_{21}&0&\cdots&0\\
f_{12}&f_{22}&0&\cdots&0\end{pmatrix}}\Lambda^{m+2}(-1)\xrightarrow{\scriptsize\begin{pmatrix}y_1&y_2&x_1&\cdots&x_m\end{pmatrix}}\Lambda.
$$
Suppose that there is some nonzero element $g\in K_{3,j}=\Lambda(-s-1)_j=\Lambda_{j-s-1}$ such that $d_2(g)=0$, i.e. $y_1g=y_2g=0$. Let us pick such an element $g$ with minimal $j$.
We need to show that there exist homogeneous $h_1,h_2\in \bD(w)$ such that $d_1(h_1,h_2)=0$ and $(h_1,h_2)\not\in\Im d_2$, i.e. $f_{11}h_1+f_{21}h_2=f_{12}h_1+f_{22}h_2=0$ and there is no $h\in \bD(w)$ such that $h_1=\nu^{-1}(y_1)h$ and $h_2=\nu^{-1}(y_2)h$, and $\deg(h_1)=\deg(h_2)\le \deg(g)$.
Suppose that the last mentioned assertion is not satisfied. We will show by induction on $i$ that for any $\phi\in \bD(w)_i^*$ there is a unique $g_{\phi}\in \bD(w)_{\deg(g)-i}$ such that $ug_{\phi}=\phi(u)g$ for any $u\in\bD(w)_i$. This will immediately yield a contradiction with the fact that $\bD(w)$ has infinite dimension.
The case $i=0$ is trivial, because $\bD(w)_0=\kk$, and hence one has $g_{\phi}=\phi(1)g$. Suppose now that we have proved the required assertion for functionals of degree $k$ less than $i$. Let us pick some $\phi\in \bD(w)_i^*\subset \bD(w)_{i-1}^*\otimes U^*$ and write it in the form $\phi=\phi_1r_1+\phi_2r_2$ with $\phi_1,\phi_2\in\bD(w)_{i-1}^*$, where $r_1,r_2$ is the basis of $U^*$ dual to $\nu^{-1}(y_1),\nu^{-1}(y_2)$. Let us prove that $f_{11}g_{\phi_1}+f_{21}g_{\phi_2}=f_{12}g_{\phi_1}+f_{22}g_{\phi_2}=0$.
If $i+s-2>\deg(g)$, then $f_{11}g_{\phi_1}+f_{21}g_{\phi_2}$ and $f_{12}g_{\phi_1}+f_{22}g_{\phi_2}$ belong to $(y_1,y_2)g=0$ by the definition of $g_{\phi_1}$ and $g_{\phi_2}$. Suppose that $i+s-2\le\deg(g)$.
By our induction hypothesis, it is enough to show that $u(f_{11}g_{\phi_1}+f_{21}g_{\phi_2})=u(f_{12}g_{\phi_1}+f_{22}g_{\phi_2})=0$ for any $u$ of degree $\deg(g)-i-s+2$. We have
\begin{multline*}
u(f_{11}g_{\phi_1}+f_{21}g_{\phi_2})=\phi_1(uf_{11})g+\phi_2(uf_{21})g=\big(\phi_1(uf_{11})r_1\nu^{-1}(y_1)+\phi_2(uf_{21})r_2\nu^{-1}(y_2)\big)g\\
=\phi\big(uf_{11}\nu^{-1}(y_1)+uf_{21}r_2\nu^{-1}(y_2)\big)g=\phi(uf_1)g=0.
\end{multline*}
The second equality can be proved in the same way. Our assumptions imply that there is some $g_{\phi}$ such that $g_{\phi_1}=\nu^{-1}(y_1)g_{\phi}$ and $g_{\phi_2}=\nu^{-1}(y_2)g_{\phi}$. Then an arbitrary element $u\in\bD(w)_i$ can be presented in the form $u=u_1\nu^{-1}(y_1)+u_2\nu^{-1}(y_2)$, and hence we get
$$
ug_{\phi}=u_1g_{\phi_1}+u_2g_{\phi_2}=\big(\phi_1(u_1)+\phi_2(u_2)\big)g=\phi\big(u_1\nu^{-1}(y_1)+u_2\nu^{-1}(y_2)\big)g=\phi(u)g.
$$
It remains to show that $g_{\phi}$ is unique that is equivalent to the fact that if $h\in \bD(w)_{\deg(g)-i}$ and $uh=0$ for any $u\in \bD(w)_{i}$, then $h=0$. Let us pick some $u$ of maximal degree such that $uh\not=0$. Then we have $y_1uh=y_2uh=0$ and $\deg(uh)<\deg(g)$ by our assumptions that contradicts the minimality assumption on the degree of $g$. Thus, we have proved the required assertion in the case where $A=A^0$ and $M={}_{\nu^{-1}}D\oplus\kk^m$.

Finally, suppose that $A=A^0$ and $M=B_2\oplus\kk^m$, i.e. $\Lambda=\kk\langle x_1,\dots,x_m,y_1,y_2\rangle/(y_1f,fy_2)$ and $f$ does not satisfy conditions listed in item 2 of Theorem \ref{A0B}. In this case the generalized Koszul complex has the form
$$
\Lambda(-s-1)\xrightarrow{\scriptsize\begin{pmatrix}y_2\\0\end{pmatrix}}\Lambda^2(-s)\xrightarrow{\scriptsize\begin{pmatrix}f&0&0&\cdots&0\\f^{y_1}y_2&f^{y_2}y_2&f^{x_1}y_2&\cdots&f^{x_m}y_2\end{pmatrix}}\Lambda^{m+2}(-1)\xrightarrow{\scriptsize\begin{pmatrix}y_1,y_2&x_1&\cdots&x_m\end{pmatrix}}\Lambda,
$$
where $f=y_1f^{y_1}+y_2f^{y_2}+\sum\limits_{i=1}^mx_if^{x_i}$.
Suppose that there is some nonzero element $g\in K_{3,j}=\Lambda(-s-1)_j=\Lambda_{j-s-1}$ such that $d_2(g)=0$, i.e. $y_2g=0$. Then $(0,g)\in K_{2,j-1}=\Lambda^2(-s)_{j-1}$ satisfies the conditions $d_1(0,g)=0$ and $(0,g)\not\in\Im(d_{2})$.
\end{proof}

Our next goal is to verify what pairs $(A,M)$ from the $s$-homogeneous triples corresponding to $s$-homogeneous algebras with two relations admit an $s$-Koszul algebra $\Lambda$ such that $(\Lambda^!)^{(s)})=A$ and $(\Lambda^!)^{(s,1)}=M$ that is equivalent to the isomorphism $\Ext_{\Lambda}^*(\kk,\kk)\cong A\ltimes M$ (see Subsection \ref{KozCY}). We will solve this question for all pairs $(A,M)$ except pairs of the form $(A^0,{}_{\nu^{-1}}D\oplus\kk^m)$, where $\nu(y_1)=\lambda y_1$ and $\nu(y_2)=\lambda^{-s}y_2$ for $\lambda\in\kk$ such that $\lambda^{s+1}$ is a primitive root of unity of degree $s-1$.
\begin{itemize}
    \item Let $A=A^1$ and $M=A/(x)\oplus\kk^m$, i.e. $\Lambda=\kk\langle x_1,\dots,x_m,y_1\rangle/(f,y_1^s)$ and $f$ does not satisfy conditions listed in Theorem \ref{A1}. It is not difficult to see that $\Lambda$ is $s$-Koszul for $f=x_1y_1^{s-2}x_2$. Note that we have given an example of $s$-Koszul $\Lambda$ in the case $m> 2$. If $m=2$ and $s=3$, then $\Lambda$ can not be $s$-Koszul, because $s$-Koszulity implies
    \begin{multline*}
    \HH_\Lambda(t)=\frac{1}{1-2t+2t^3+(t^{6}-t^{4})\sum\limits_{i=0}^\infty t^{3i}}=\frac{1-t^3}{1-2t+t^3+t^4-t^{6}}\\
    =1 + 2 t + 4 t^2 + 6 t^3 + 9 t^4 + 12 t^5 + 15 t^6 + 17 t^7 + 17 t^8 + 13 t^9 + 3 t^{10} - 16 t^{11} + O(t^{12})
    \end{multline*}
    that is impossible. If $m=2$ and $s>3$, then $\Lambda$ is $s$-Koszul for $f=y_1x_1^{s-1}-x_1y_1^{s-2}x_1$. To show this, let us choose the lexicographical order on monomials induced by the order $y_1>x_1$. Then one can show that the elements $y_1^s$, $y_1x_1^{s-1}-x_1y_1^{s-2}x_1$ and $y_1^{s-1}(x_1y_1^{s-3})^ky_1x_1$ ($k\ge 1$) constitute a Gr\"obner of the ideal $(y_1^s,f)$. Now one can directly show that the generalized Koszul complex is exact in the second degree or calculate the Hilbert series of the algebra $\bar\Lambda=\kk\langle x_1,y_1\rangle/(y_1^s, y_1x_1^{s-1}-x_1y_1^{s-2}x_1, \{y_1^{s-1}(x_1y_1^{s-3})^ky_1x_1\}_{k\ge 1})$ using the Anick resolution and get that $$\HH_{\Lambda}(t)=\HH_{\bar\Lambda}(t)=\frac{1}{1-2t+2t^s+(t^{2s}-t^{s+1})\sum\limits_{i=0}^\infty t^{si}}=\frac{1-t^s}{1-2t+t^s+t^{s+1}-t^{2s}}.$$
    
    \item Let $A=A^0$ and $M={}_{\nu^{-1}}D\oplus\kk^m$, where $\nu$ is an automorphism satisfying one of the conditions of Lemmas \ref{CYs} and \ref{CYns}, i.e. $\Lambda=\kk\langle x_1,\dots,x_m\rangle*\bD(w)$ for some $\nu$-twisted potential $w$ on a two dimensional space with the basis $y_1,y_2$ that does not satisfy conditions listed in Theorem \ref{A0Z}.
    In this case $\Lambda$ is $s$-Koszul if and only if the algebra $\bD(w)$ is $\nu$-twisted $3$-Calabi-Yau.
    
    If $\nu(y_1)=\lambda y_1$ and $\nu(y_2)=\lambda(y_1+y_2)$ for some $\lambda\in\kk$ such that either $s>3$ and $\lambda^{s+1}=1$ or $s=3$ and $\lambda^2=1$, then the twisted potential $w$ from the proof of Lemma \ref{CYns} defines a twisted $3$-Calabi-Yau algebra. Indeed, $(f_1\otimes 1^{\otimes s})w$ and $(f_2\otimes 1^{\otimes s})w$, where $f_1,f_2$ is the basis of $V^*$ dual to the basis $y_1,y_2$, have largest monoms $y_2^2y_1^{s-2}$ and $y_2y_1^{s-1}$ respectively. Here we use the lexicographical order induced by the order $y_2>y_1$. Then one can check that these elements constitute a Gr\"obner basis of the ideal $\big((f_1\otimes 1^{\otimes s})w,(f_2\otimes 1^{\otimes s})w\big)$, and hence $\Lambda$ has the required Hilbert series.

The case $\nu(y_1)=\lambda_1 y_1$ and $\nu(y_2)=\lambda_2 y_2$ with $\lambda_1^k\lambda_2^{s+1-k}=1$ for some $2\le k\le s-1$ can be considered in the same way using the potential defined in the proof of Lemma \ref{CYs}. In this case $(f_1\otimes 1^{\otimes s})w$ and $(f_2\otimes 1^{\otimes s})w$ have largest monoms $y_2^{s+1-k}y_1^{k-1}$ and $y_2^{s-k}y_1^{k}$ respectively.
    
    Suppose now that  $\lambda_1=1$, $\lambda_2^s=1$. Then the potential $w=y_1^{s+1}+\sum\limits_{i=0}^{s}\lambda_2^iy_2^iy_1y_2^{s-i}$ defines a $\nu$-twisted $3$-Calabi-Yau algebra. Really, we have $\bD(w)=\kk\langle y_1,y_2\rangle/(y_1^s+y_2^s,\sum\limits_{i=0}^{s-1}\lambda_2^iy_2^iy_1y_2^{s-i-1})$. Setting $y_1>y_2$ and choosing the corresponding lexicographical order on monoms, we get the Gr\"obner basis $$y_1^s+y_2^s,\,\,\sum\limits_{i=0}^{s-1}\lambda_2^iy_2^iy_1y_2^{s-i-1},\,\,\sum\limits_{i=0}^{s-2}\lambda_2^iy_1^{s-1}y_2^{i+1}y_1y_2^{s-i-2}-\lambda_2^{s-1}y_2^{2s-1}.$$ Then the Hilbert series of $\bD(w)$ equals the Hilbert series of the algebra $\kk\langle y_1,y_2\rangle/(y_1^s, y_1y_2^{s-1},y_1^{s-1}y_2y_1y_2^{s-2})$. It can be easily verified that the Hilbert series of the last mentioned algebra is $\frac{1}{1-2t+2t^s-t^{s+1}}$.

Finally, in the case $\lambda_1\lambda_2^s=\lambda_1^s\lambda_2=1$ and $\lambda_1^k\lambda_2^{s+1-k}\not=1$ for any $2\le k\le s-1$ one has $\lambda_1=\lambda$ and $\lambda_2=\lambda^{-s}$ for some $\lambda\in\kk$ such that $\lambda^{s+1}$ is a primitive root of unity of degree $s-1$.
Due to the proof of Lemma \ref{CYs}, the algebra $\Lambda^!$ has an $s$-homogeneous triple with the first two components $(A^0,{}_{\nu^{-1}}D\oplus\kk^m)$ if and only if $\Lambda\cong \kk\langle x_1,\dots,x_m,\rangle*\bD(w)$ with $w=\sum\limits_{i=0}^{s+1}\lambda^iy_1^iy_2y_1^{s-i}+\sum\limits_{i=0}^{s+1}\lambda^{-is}y_2^iy_1y_2^{s-i}$. Thus, the algebra $A^0\ltimes \big({}_{\nu^{-1}}D\oplus\kk^m\big)$ can be $\Ext$-algebra of some $s$-Koszul algebra if and only if the algebra
$$\bD(w)=\frac{\kk\langle y_1,y_2\rangle}{\left(y_2^s+\sum\limits_{i=0}^{s}\lambda^{i+1}y_1^{i}y_2y_1^{s-i-1},y_1^s+\sum\limits_{i=0}^{s}\frac{y_2^{i}y_1y_2^{s-i-1}}{\lambda^{(i+1)s}}\right)}$$
is $s$-Koszul. The verification of this is a tedious task, because the Gr\"obner basis of this algebra grows fast while the degree increase. We believe that $\bD(w)$ is an $s$-Koszul (and, in particular, twisted CY) algebra with infinite Gr\"obner basis, but the proof of this is out of this work.
    
    \item Let $A=A^0$ and $M=B_2\oplus\kk^m$, i.e. $\Lambda=\kk\langle x_1,\dots,x_m,y_1,y_2\rangle/(y_1f,fy_2)$ and $f$ does not satisfy conditions listed in item 2 of Theorem \ref{A0B}. It is not difficult to show that $\Lambda$ is $s$-Koszul for $f=y_1y_2^{s-2}$.
    
    \item Let $A=A^0$ and $M=\kk^m$, i.e. $\Lambda=\kk\langle x_1,\dots,x_m\rangle/(f_1,f_2)$, where and $f_1$ and $f_2$ do not satisfy conditions listed in Theorem \ref{A0s}.  It is not difficult to show using  Gr\"obner bases that $\Lambda$ is $s$-Koszul for $f_1=x_1^{s-1}x_2$, $f_2=x_1^{s-1}x_3$. If $s\ge 5$, then $\Lambda$ is $s$-Koszul also for $f_1=x_1x_2x_1x_2^{s-3}$, $f_2=x_1^2x_2^{s-2}$. Thus, we have shown that $\Lambda$ can be $s$-Koszul if $m\ge 3$ or $s\ge 5$. If $m=2$ and $s=3,4$, then $\Lambda$ is $s$-Koszul if and only if its Hilbert series equals
    $$
    \HH_\Lambda(t)=\frac{1}{1-2t+2t^3}=1 + 2 t + 4 t^2 + 6 t^3 + 8 t^4 + 8 t^5 + 4 t^6 - 8 t^7  + O(t^{8})
    $$
    or
    \begin{multline*}
        \HH_\Lambda(t)=\frac{1}{1-2t+2t^4}=1 + 2 t + 4 t^2 + 8 t^3 + 14 t^4 + 24 t^5 + 40 t^6 + 64 t^7 + 100 t^8 + 152 t^9 + 224 t^{10}\\
        + 320 t^{11} + 440 t^{12} + 576 t^{13} + 704 t^{14} + 768 t^{15} + 656 t^{16} + 160 t^{17} - 1088 t^{18} + O(t^{19})
    \end{multline*}
    that is impossible.
\end{itemize}

\section{Main result, consequences and open problems}

In this section, for the convenience of the reader, we give a summary  of our classification.

\subsection{The main result}\label{main_res}
Here we collect the results of Theorems \ref{A4A5}, \ref{A6A70}, \ref{impos_struc}, \ref{A7q}, \ref{A1}, \ref{A0Z}, \ref{A0B}, and \ref{A0s} to get the following classification.

Let $\Lambda=T_\kk V/(f_1, f_2)$, where $f_1$ and $f_2$ are two linearly independent elements of $V^{\ot s}$. Let us denote by $A$ the algebra $(\Lambda^!)^{(s)}$ and by $M$ the $A$-bimodule $(\Lambda^!)^{(s,1)}$. Then modulo isomorphism one of the following conditions holds:
\begin{enumerate}
    \item $(A,M)=\big(\kk\langle x,y\rangle/(xy,yx), J(A)[1]\oplus\kk^m\big),\Lambda\cong\kk\langle x_1,\dots,x_m,y_1,y_2\rangle/(y_1^s,y_2^s)$.
    \item $(A,M)=\left(\kk\langle x,y\rangle/(xy,yx), {}_{\sigma}J(A)[1]\oplus\kk^m\right)$, where $\sigma$ is the automorphism of $A$ interchanging $x$ and $y$. In this case $s=2t$ and
    $$\Lambda\cong\kk\langle x_1,\dots,x_m,y_1,y_2\rangle/\big((y_1y_2)^t,(y_2y_1)^t\big).$$
    \item $(A,M)=\left(\kk\langle x,y\rangle/(x^2,y^2), {}_{\sigma}J(A)[1]\oplus\kk^m\right)$, where $\sigma$ is the automorphism of $A$ interchanging $x$ and $y$. In this case $s=2t+1$ and
    $$\Lambda\cong\kk\langle x_1,\dots,x_m,y_1,y_2\rangle/\big((y_1y_2)^ty_1,(y_2y_1)^ty_2\big).$$
    \item $(A,M)=\big(\kk\langle x,y\rangle/(x^2,yx), J(A)[1]\oplus\kk^m\big),\Lambda\cong\kk\langle x_1,\dots,x_m,y_1,y_2\rangle/(y_1^s,y_2y_1^{s-1})$.
    \item $(A,M)=\big(\kk\langle x,y\rangle/(x^2,xy), J(A)[1]\oplus\kk^m\big),\Lambda\cong\kk\langle x_1,\dots,x_m,y_1,y_2\rangle/(y_1^s,y_1^{s-1}y_2)$.
    \item $(A,M)=\big(\kk\langle x,y\rangle/(x^2,xy-qyx), M(p)\oplus\kk^m\big)$, where $p,q\in\kk^*$, $p^s=q$ and $M(p)$ is defined in Subsection \ref{A7sec}. In this case
    $$\Lambda\cong\kk\langle x_1,\dots,x_m,y_1,y_2\rangle/\left(y_1^s,\sum\limits_{i=0}^{s-1}p^{s-i-1}y_1^iy_2y_1^{s-i-1}\right).$$
    \item $(A,M)=\big(\kk\langle x,y\rangle/(x^2,xy,yx), A/(x)\oplus\kk^m\big)$. Such pairs correspond to algebras of the form $\Lambda\cong\kk\langle x_1,\dots,x_m,y_1\rangle/\left(y_1^s,f\right)$ with $f\in \kk\langle x_1,\dots,x_m,y_1\rangle_s$, where $f\not=\sum\limits_{i=0}^{s-1}\alpha^i\beta^{s-i-1}y_1^igy_1^{s-i-1}$ and $f\not=g^s+\alpha y_1^s$ for any linear polynomial $g$ and $\alpha,\beta\in\kk$.
    \item $(A,M)=\big(\kk\langle x,y\rangle/(x,y)^2, {}_{\nu^{-1}}D\oplus\kk^m\big)$, where $D$ is the bimodule defined in Subsection \ref{CY} and $\nu$ is a graded automorphism of $\kk\langle x,y\rangle/(x,y)^2$ defined in one of the following ways:
\begin{enumerate}
\item\label{gen} $\nu(x)=\lambda_1x$, $\nu(y)=\lambda_2y$, where $\lambda_1^k\lambda_2^{s+1-k}=1$ for some $2\le k\le s-1$;
\item\label{1ss1} $\nu(x)=\lambda_1x$, $\nu(y)=\lambda_2y$, where $\lambda_1\lambda_2^s=\lambda_1^s\lambda_2=1$ and $\lambda_1^k\lambda_2^{s+1-k}\not=1$ for any $2\le k\le s-1$;
\item\label{1s} $\nu(x)=\lambda_1x$, $\nu(y)=\lambda_2y$, where $\lambda_1=1$, $\lambda_2^{s}=1$ and $\lambda_2^k\not=1$ for any $1\le k\le s-1$;
\item\label{genns} $\nu(x)=\lambda x$, $\nu(y)=\lambda (x+y)$, where $\lambda^{s+1}=1$ if $s>3$ and $\lambda=\pm 1$ if $s=3$.
\end{enumerate}
    Such pairs correspond to algebras of the form $\Lambda\cong\kk\langle x_1,\dots,x_m\rangle*\bD(w)$ with a $\nu$-twisted potential $w\in\kk\langle y_1,y_2\rangle_{s+1}$ on the two dimensional space $U$ with the basis $y_1,y_2$, where $w\not= g_1g^s+g_2f$, $w\not= (g_1g_2)^{t+1}+\alpha(g_2g_1)^{t+1}$, and $w\not=(g_1g_2)^tg_1+\alpha(g_2g_1)^tg_2$ for any $g_1,g_2,g\in U$, $f\in \kk\langle y_1,y_2\rangle_s$ and $\alpha\in\kk^*$.
    Here $\nu:U\rightarrow U$ is the composition $U\xrightarrow{\theta}A_1\xrightarrow{\nu}A_1\xrightarrow{\theta^{-1}}U$ with the isomorphism $\theta:U\rightarrow A_1$ defined by the equalities $\theta(y_1)=x$ and $\theta(y_2)=y$.
In the cases \ref{gen} and \ref{genns} one has as minimum one algebra $\Lambda$ such that $(\Lambda^!)^{(s)}=A_0$ and $(\Lambda^!)^{(s,1)}={}_{\nu^{-1}}D\oplus\kk^m$. In the cases \ref{1ss1} and \ref{1s} one has exactly one such an algebra. Namely, in the case \ref{1ss1} the unique $\nu$-twisted potential algebra satisfying the required properties is
$$\Lambda=\kk\langle x_1,\dots,x_m\rangle*\frac{\kk\langle y_1,y_2\rangle}{\left(y_2^s+\sum\limits_{i=0}^{s}\lambda_1^{i+1}y_1^{i}y_2y_1^{s-i-1},y_1^s+\sum\limits_{i=0}^{s}\lambda_2^{i+1}y_2^{i}y_1y_2^{s-i-1}\right)}$$
while in the case \ref{1s} the unique algebra  satisfying the required properties is
$$
\Lambda=\kk\langle x_1,\dots,x_m\rangle*\frac{\kk\langle y_1,y_2\rangle}{\left(y_1^s+y_2^s,\sum\limits_{i=0}^{s-1}\lambda_2^iy_2^iy_1y_2^{s-i-1}\right)}.
$$
    \item $(A,M)=\big(\kk\langle x,y\rangle/(x,y)^2, B_1\oplus\kk^m\big)$, where $B_1$ is the bimodule defined in Subsection \ref{Bsec}. Such pairs correspond to algebras of the form $\Lambda\cong\kk\langle x_1,\dots,x_m,y_1\rangle/(y_1f,fy_1)$ with $f\in \kk\langle x_1,\dots,x_m,y_1\rangle_{s-1}$, where $f\not=(gy_1)^{t-1}g$ for any linear polynomial $g$.
    \item $(A,M)=\big(\kk\langle x,y\rangle/(x,y)^2, B_2\oplus\kk^m\big)$, where $B_2$ is the bimodule defined in Subsection \ref{Bsec}. Such pairs correspond to algebras of the form $\Lambda\cong\kk\langle x_1,\dots,x_m,y_1,y_2\rangle/(y_1f,fy_2)$ with $f\in \kk\langle x_1,\dots,x_m,y_1,y_2\rangle_{s-1}$, where $f\not=\alpha(y_2y_1)^t$, $f\not=\alpha y_1^{s-1}$, and $f\not=\alpha y_2^{s-1}$ for any $\alpha\in\kk^*$.
    \item $(A,M)=\big(\kk\langle x,y\rangle/(x,y)^2, \kk^m\big)$. Such pairs correspond to all algebras that do not appear in one of the previous items. In more details this condition is written down in Theorem \ref{A0s}.
\end{enumerate}

 Let us now collect the results on the $s$-Koszulity of $s$-homogeneous algebras with two relations.

\begin{itemize}
    \item If one of Conditions 1--6 above is satisfied, then $\Lambda$ is an $s$-Koszul algebra.
    \item If Conditions 7 is satisfied, then $\Lambda$ can be $s$-Koszul in the case where $m>2$ or $s>3$. In the case $(m,s)=(2,3)$ the algebra $\Lambda$ is not $s$-Koszul.
    \item If Conditions 8 is satisfied, then $\Lambda$ can be $s$-Koszul in the cases \ref{gen}, \ref{1s} and \ref{genns}. It is not known if $\Lambda$ is $s$-Koszul or not in the case \ref{1ss1}
    \item if Condition 9 above is satisfied, then $\Lambda$ does not satisfy the extra condition, and hence it is not $s$-Koszul.
    \item If Conditions 10 is satisfied, then $\Lambda$ can be $s$-Koszul for any $m\ge 2$ and $s\ge 3$.
    \item If Conditions 11 is satisfied, then $\Lambda$ can be $s$-Koszul in the case where $m>2$ or $s>4$. In the cases $(m,s)=(2,3)$ and $(m,s)=(2,4)$ the algebra $\Lambda$ is not $s$-Koszul.
    \item If one of Conditions 7, 8, 10, 11 is satisfied, then the $s$-Koszulity of $\Lambda$ is equivalent to the exactness of the generalized Koszul complex in the second term and, at the same time, to the equality $\HH_{\Lambda}(t)=\Big(\HH_A(t^s)-t\HH_M(t^s)\Big)^{-1}$, where the right hand side of the last mentioned equality can be easily calculated in each case.
\end{itemize}

Note that if $\Lambda$ is $s$-Koszul, then, as it was mentioned above, $\Ext_{\Lambda}^*(\kk,\kk)$ is the trivial extension of $A$ by $M$ with an appropriate grading. Thus, we have described a set of possible $\Ext$-algebras of connected $s$-Koszul algebras with two relations and show that all of these algebras are possible except the algebras arising in the case \ref{1ss1}.

\subsection{Some consequences and problems}

Here we mention some results that follow from our classification and state some problems that seem to be interesting for us.

Note first of all that the number of different quadratic algebras that can appear as first components of an $s$-homogeneous triple $\FF(\Lambda^!)$ is very restricted: only the algebras $A^0$, $A^1$, $A^4$, $A^5$, $A^6$ and $A^7(q)$ ($q\in\kk$) are possible. Moreover, all of this algebras are Koszul and have minimum two relations.

\begin{Ques} Given a quadratic algebra $A$, what are the sufficient and necessary conditions for $A$ to be isomorphic to $\Lambda^{(s)}$ for some $s$-homogeneous algebra $\Lambda$? Does $A$ have to be Koszul? In the case where the number of generators of $A$ is given, how many relations can $A$ have?
\end{Ques}

As it was stated above, an $\Ext$-algebra of an $s$-Koszul algebra with two relations is isomorphic to $A\ltimes M$, where the pair $(A,M)$ is described in one of Conditions 1--11 of the previous subsection. For each such a pair it is determined if it really can be $\Ext$-algebra with only one exception. This exception motivates the next question.

\begin{Ques} Given $\lambda\in\kk$ such that $\lambda^{s+1}$ is a primitive root of unity of degree $s-1$, is the algebra
$$\bD(w)=\frac{\kk\langle y_1,y_2\rangle}{\left(y_2^s+\sum\limits_{i=0}^{s}\lambda^{i+1}y_1^{i}y_2y_1^{s-i-1},y_1^s+\sum\limits_{i=0}^{s}\frac{y_2^{i}y_1y_2^{s-i-1}}{\lambda^{(i+1)s}}\right)}$$
$s$-Koszul?
\end{Ques}

If the answer to this problem is positive, then we will get the following result. If there exists an $s$-homogeneous algebra $\Lambda$ such that $\big((\Lambda^!)^{(s)}),(\Lambda^!)^{(s,1)}\big)=(A^0,{}_{\nu^{-1}}D)$, then there exists an $s$-Koszul algebra $\Gamma$ such that $\big((\Gamma^!)^{(s)}),(\Gamma^!)^{(s,1)}\big)$. This assertion has a generalization to an arbitrary number of relations.

\begin{Ques} Suppose that there exists a $\nu$-twisted potential $w$ on an $n$-dimensional space $V$ such that the generalized Koszul complex of $\bD(w)$ has the form $$\bD(w)\rightarrow\bD(w)^n\rightarrow\bD(w)^n\rightarrow\bD(w)(\rightarrow\kk).$$
Does it follows from this that there exists a $\nu$-twisted potential $w'$ such that $\bD(w')$ is twisted $3$-CY?
\end{Ques}

In our classification, the algebras are divided into several concrete algebras and several series of algebras. It would be interesting to give for each of this series a criterion for $s$-Koszulity like in Theorem \ref{MV_theor}.

\begin{Ques} For the pairs $(A,M)$ satisfying one of the Conditions 7, 8, 10 and 11, give some easily verifiable criterion of $s$-Koszulity.
\end{Ques}

Our result allow to compute all possible Hilbert series of $s$-Koszul algebras with two relations that can be computed using \eqref{HSm}. Namely, the possible Hilbert series are:
\begin{multline*}
\frac{1-t^s}{1-mt+t^s+(m-2)t^{s+1}},\,\,\frac{1-t^s}{1-mt+t^s+(m-1)t^{s+1}-t^{2s}}\,\,(\mbox{$s>3$ or $m>2$}),\\
\frac{1}{1-mt+2t^s-t^{s+1}},\,\,\frac{1}{1-mt+2t^s}\,\,(\mbox{$s>4$ or $m>2$}).
\end{multline*}
All of these series are rational functions and this lead to the next problem (the same problem for Koszul algebras is stated in \cite[Conjecture 1]{PP}).

\begin{Ques} Is the Hilbert series rational for any $s$-Koszul algebra?
\end{Ques}

All the Hilbert and Poincar\'e series of $s$-homogeneous algebras with one relation were computed in \cite{Dic}.

\begin{Ques} Find all the possible Hilbert and Poincar\'e series of $s$-homogeneous algebras with two relations.
\end{Ques}

The computations of \cite{Dic}, in particular, imply that the global dimension of an $s$-homogeneous algebra with one relation can be two of infinity. Our results show that the global dimension of an $s$-Koszul algebra with two relations can be two, three or infinity.

\begin{Ques}
 Is the global dimension of an $s$-homogeneous restricted by some function depending on the number of relations in the case where it is finite?
Is the number of possible Hilbert Series of $s$-homogeneous algebras finite for given numbers of generators and relations? Are these assertion true for $s$-Koszul algebras in the case of a negative answer for arbitrary $s$-homogeneous algebras?
\end{Ques}

Finally, we have shown that if the $s$-homogeneous algebra $\Lambda$ has $m$ generators and $m^s-2$ relations, then it cannot be $s$-Koszul. On the other hand, if $s>4$ or $m>2$, then a generic $s$-homogeneous algebra with $m$ generators and two relations is $s$-Koszul. One can show that in the case where $s=3,4$ and $m=2$, a generic algebra is not $s$-Koszul.

\begin{Ques} For which triples $(s,m,k)$ of integer numbers there exists an $s$-Koszul algebra with $m$ generators and $k$ relations? For which triples $(s,m,k)$ of integer numbers a generic algebra with $m$ generators and $k$ relations is $s$-Koszul?
\end{Ques}

\noindent{{\bf Addresses:}
\newline 
Eduardo do Nascimento Marcos\\
IME-USP(Dept de Matem\'atica )
\newline
Rua de Mat\~ao 1010, Cidade Universit\'aria, S\~ao Paulo-SP
\newline
e-mail:  enmarcos@ime.usp.br
\newline
Yury Volkov \\
Saint-Petersburg State University\\
Universitetskaya nab. 7-9, St. Peterburg, Russia\\
e-mail:  wolf86\_666@list.ru}

\end{document}